\documentclass{AIMS} 

\usepackage{amsmath,amsthm,amssymb,mathtools}
\usepackage{mathrsfs}
\usepackage{graphicx}
\usepackage{comment}
\usepackage{enumerate}

\usepackage[english]{babel} 
\usepackage{lmodern}
\usepackage[T1]{fontenc}
\usepackage[utf8]{inputenc} 

\usepackage{graphics}            
\graphicspath{{Figures/}}        
\usepackage{import}              
 \usepackage{xcolor}             
\usepackage{graphicx}            
\usepackage{psfrag}              

\usepackage{changebar}
\usepackage{todonotes}
\usepackage{changes}
\definechangesauthor[name={Marco}, color=orange]{MA}
\definechangesauthor[name={UMA}, color=blue]{UMA}
\def\startmodif{\begingroup\color{black}}
\def\stopmodif{\endgroup}

\usepackage[toc,page]{appendix}

\usepackage{tikz}
\newcommand\encircle[1]{%
  \tikz[baseline=(X.base)] 
    \node (X) [draw, shape=circle, inner sep=0] {\strut #1};}

\usepackage[colorlinks=true]{hyperref} 

\def\currentvolume{X}
 \def\currentissue{X}
  \def\currentyear{200X}
   \def\currentmonth{XX}
    \def\ppages{X--XX}
     \def\DOI{10.3934/xx.xx.xx.xx}

\renewcommand{\le}{\leqslant}
\renewcommand{\ge}{\geqslant}
 
\renewcommand{\geq}{\geqslant}

\newcommand{\lang}{\left\langle}
\newcommand{\rang}{\right\rangle}

\newcommand{\R}{\mathbb{R}}

\renewcommand{\L}{\operatorname{L}}

\makeatletter
\newcommand{\@syst}[1]{%
\[ 
\left\{
\begin{aligned}
#1
\end{aligned}
\right.
\]
}

\newcommand{\@systn}[1]{%
\begin{equation}
\left\{
\begin{aligned}
#1
\end{aligned}
\right.
\end{equation}
}

\makeatother

\newcommand{\SO}{{SO}} 
\newcommand{\controlset}{{\mathcal{U}}} 

\newcommand{\FF}{F} 
\newcommand{\GG}{G} 
\newcommand{\UU}{U} 
\newcommand{\qq}{X} 
\newcommand{\pp}{P} 
\newcommand{\px}{{p}} 
\newcommand{\py}{{q}} 

\newcommand{\umax}{u_{\max}}
\newcommand{\vmin}{v_{\min}}
\newcommand{\vmax}{v_{\max}}
\newcommand{\rmin}{r_{\min}}
\newcommand{\pzero}{\lambda}
\newcommand{\phiu}{\phi_u} 
\newcommand{\phiv}{\phi_v} 
\newcommand{\uv}{(u,v)}
\newcommand{\DA}{\Delta_{A}}
\newcommand{\DBu}{\Delta_{Bu}}
\newcommand{\DBv}{\Delta_{Bv}}

\newcommand{\Pdub}{\bf ($\mathbf{P_{\!0}}$)} 
\newcommand{\Pone}{\bf ($\mathbf{P_{\!1}}$)}
\newcommand{\Ptwo}{\bf ($\mathbf{P_{\!2}}$)}
\newcommand{\Pgen}{\bf ($\mathbf{R}$)}

\newcommand{\qzero}{X_0}
\newcommand{\qfinal}{{X_f}}

\newcommand{\xt}{\tilde x}
\newcommand{\yt}{\tilde y}
\newcommand{\Xzerot}{{\tilde X_0}}

\newcommand{\qqt}{\tilde X}

\newcommand{\tsing}{{t_{\mathrm{sing}}}} 
\newcommand{\tsw}{\tau} 
\newcommand{\Interval}{t_0,t_1} 

\newcommand{\target}{\mathcal{C}} 

\newcommand{\UAV}{{UAV }}
\newcommand{\drone}{{drone }}

\newcommand{\cm}{m}
\newcommand{\cp}{p}
\newcommand{\cM}{M}
\newcommand{\cP}{P}
\newcommand{\cs}{s}

\newcommand{\traj}[1]{\gamma^{#1}}
\newcommand{\Tsw}[1]{T^{#1}}
\newcommand{\zphiu}[1]{\zeta^{#1}_{\phiu}}
\newcommand{\zphiv}[1]{\zeta^{#1}_{\phiv}}
\newcommand{\xtswc}[1]{\xt_{\mathrm{sw}}^{#1}}
\newcommand{\ytswc}[1]{\yt_{\mathrm{sw}}^{#1}}

\newcommand{\VM}{\p} 
\newcommand{\vm}{1} 
\newcommand{\VMvmVM}{\frac{\p+1}{\p}} 
\newcommand{\VMvmVMcos}{\VMvmVM\cos\alpha}

\newcommand{\asing}{\alpha_{\mathrm{sing}}}
\newcommand{\sqrtacos}{\sqrt{1-\left(\VMvmVM\cos\alpha\right)^2}}
\newcommand{\acos}{\arccos\left(\VMvmVM\cos\alpha\right)}
\newcommand{\asin}{\arcsin\left(\VMvmVM\cos\alpha\right)}
\newcommand{\atraj}[1]{\alpha^{#1}}

\newcommand{\VMCmvmCVMcos}{\frac{(\p^2-1)}{\p}\cos\alpha}

\newcommand{\p}{\eta} 

\newtheorem{theorem}{Theorem}[section]

\newtheorem{lemma}[theorem]{Lemma}
\newtheorem{proposition}[theorem]{Proposition}

\theoremstyle{definition}
\newtheorem{definition}[theorem]{Definition}

\newtheorem{remark}[theorem]{Remark}

\theoremstyle{definition}

\numberwithin{equation}{section} 
\title[Minimal time synthesis for a kinematic drone model]{Minimal time synthesis for a kinematic drone model}

\author[M.-A. Lagache, U. Serres and V. Andrieu]{}


 \email{marc-aurele-lagache@etud-univ-tln.fr}
 \email{ulysse.serres@univ-lyon1.fr}
 \email{vandrieu@lagep.univ-lyon1.fr}

\thanks{The authors were partially supported by the Grant ANR-12-BS03-0005 LIMICOS of the ANR.}

\thanks{$^*$ Corresponding author.}

\begin{document}

%
%
\begin{abstract}
In this paper,
we consider a (rough) kinematic model for a UAV
flying at constant altitude moving forward with positive lower and upper bounded linear velocities and positive minimum turning radius.
For this model,
we consider the problem of minimizing the time travelled by the UAV starting from a
general configuration to connect a specified target being a fixed circle of minimum turning radius.
The time-optimal synthesis is presented as
a partition of the state space which defines a unique optimal path such that the target can be reached optimally.
\end{abstract}

\maketitle

\centerline{\scshape Marc-Aur\`ele Lagache}
\medskip

{\footnotesize
 \centerline{Universit\'e de Toulon,}
   \centerline{CNRS, LSIS, UMR 7296,}
   \centerline{F-83957 La Garde, France}
\smallskip
\centerline{\&}
\smallskip
\centerline{Univ. Lyon,}
 \centerline{Universit\'e Claude Bernard Lyon 1,}
   \centerline{CNRS, LAGEP, UMR 5007}
   \centerline{43 bd du 11 novembre 1918, }
   \centerline{F-69100 Villeurbanne, France}
}
%


\medskip

\centerline{\scshape Ulysse Serres$^*$ and Vincent Andrieu}
\medskip
{\footnotesize
\centerline{Univ. Lyon,}
 \centerline{Universit\'e Claude Bernard Lyon 1,}
   \centerline{CNRS, LAGEP, UMR 5007}
   \centerline{43 bd du 11 novembre 1918, }
   \centerline{F-69100 Villeurbanne, France}
} 

\bigskip

 \centerline{(Communicated by the associate editor name)}



%
%
\section{Introduction}\label{sec:Intro}

The purpose of this study is to determine the fastest way (in time) to steer
a kinematic \UAV (or \drone\!\!) flying at a constant altitude from any starting point to
a fixed horizontal circle of minimum turning radius.

The problem is only described from a kinematic point of view.
In particular, we do not take into account the inertia of the \drone.
We consider that the \drone velocities are controlled parameters.
In consequence, they are allowed to vary arbitrarily fast.

From the kinematic point of view, a rough \drone that flies at a constant altitude
is governed by the standard Dubins equations
(see e.g. \cite{BalluchiCDC2000}, \cite{Dubins1957}):
\begin{equation}\label{eq:dubins}
  \left\{
  \begin{aligned}
      \dot{x} &= v\cos\theta \\
      \dot{y} &= v\sin\theta \\
      \dot{\theta} &= u,
  \end{aligned}
  \right.
\end{equation}
with
$(x,y,\theta) \in{\R}^2\times {\mathbb{S}}^1$ being the state
(where $(x,y) \in \R^2$ is the UAV's coordinates in the constant altitude plane,
and $\theta$ the yaw angle),
and
$u\in[-\umax,\umax]$,
$v\in[\vmin,\vmax]$ being the control variables.
Note that the yaw angle $\theta$ is the angle between 
the aircraft direction and the $x$-axis.

\startmodif
In this paper, only system \eqref{eq:dubins} is studied. Despite the fact that our motivation is UAVs, one could apply our results to other control problems modelled by system \eqref{eq:dubins}.
\stopmodif

These equations express that the \drone moves on a perfect plane (perfect constant altitude)
in the direction of its velocity vector, and is able to turn right and left.

We assume that the controls on the \drone kinematics are
its angular velocity $u$ and its linear velocity $v$.

Moreover, we make the assumptions 
that the linear velocity $v$ has a positive lower bound $\vmin$
and a positive upper bound $\vmax$ and that
the time derivative $u$ of the \drone yaw angle is 
constrained by an upper positive bound $\umax$.

The above assumptions imply in particular that
no stationary or quasi-stationary flights are allowed
and
that
the \drone is kinematically restricted by its minimum turning radius $\rmin={\vmin}/{\umax}>0$.

A similar problem with a constant linear velocity has already been addressed in \cite{MaillotBoscainGauthierSerres2015}.
The purpose of this paper is to study the influence of a non-constant linear velocity. \startmodif
A preliminary version of this work as been published in \cite{LagacheSerresAndrieuCDC15}.
\stopmodif

%
%
\section{Minimum time problem under consideration}\label{sec:drone-problem}

%
%
\subsection{Optimal control problem}

We aim to steer a \UAV driven by system (\ref{eq:dubins}) in minimum time
from any given initial position point to the target manifold $\target$ which is defined to be the counterclockwise-oriented circular trajectory of minimum turning radius centered at the origin.
In the $(x,y,\theta)$-coordinates,  $\target$ is given by
\[
\target
=\left\{ \left(x,y,\theta\right) \mid x=\rmin\sin\theta,\, y=-\rmin\cos\theta\right\}.
\]
More precisely, we consider the following optimal control problem:
\begin{itemize}
  \item[\textrm{\Pdub}]
  For every $(x_0,y_0,\theta_0)\in\R^2\times \mathbb{S}^1$ find a pair trajectory-control  joining $(x_0,y_0,\theta_0)$ to $\target$, which is  time-optimal for the control system (\ref{eq:dubins}).
\end{itemize}

%
%
\subsection{Existence of solutions}

The following two propositions are well-known and stated without proof (see, e.g. \cite{Agrachev-Sachkov-2004}).
\begin{proposition}[\startmodif
Controllability\stopmodif]
System (\ref{eq:dubins}) is controllable provided that $u\in[-\umax,\umax]$ and $v\in[\vmin,\vmax]$ for any choice of $0<\umax\le+\infty$ and $0<\vmin\le\vmax\le+\infty$.
\end{proposition}
Also, Filippov's theorem 
gives.
\begin{proposition}[Existence of minimizers]
For any point $(x_0,y_0,\theta_0)\in\R^2\times \mathbb{S}^1$, there exists a time-optimal trajectory
joining $(x_0,y_0,\theta_0)$ to $\target$.
\end{proposition}

%
%
\subsection{Dimension reduction of the system}

To solve problem \textrm{\Pdub} it is convenient to work with a reduced system in dimension two.
Indeed, in dimension two,
a complete theory to build time-optimal syntheses exists and will be described in Section \ref{sec:theory}.

Let the control set be defined by $\controlset = [-\umax,\umax]\times[\vmin,\vmax]\subset\R^2$.

Also,
we introduce the UAV-based coordinates $\left(\xt,\yt,\theta\right)$ with $\xt$ and $\yt$ defined by the transformation (in $\SO(2)$):
\begin{equation*}
  \begin{pmatrix}
    \xt \\
    \yt
  \end{pmatrix}
  =
  \begin{pmatrix}
    \cos\theta & \,\sin\theta \\
    -\sin\theta & \,\cos\theta \\
  \end{pmatrix}
  \begin{pmatrix}
    x \\
    y
  \end{pmatrix}.
\end{equation*}

The main advantage of this \UAV\!\!-based coordinate system is that it decouples the variable $\theta$
and projects the final manifold $\target$ to the point $\Xzerot=(0,-\rmin)$.
Therefore,
the original time-optimal control problem can be equivalently reformulated in the reduced state space $(\xt,\yt)$ as the following minimum-time problem:

\begin{itemize}
  \item[{\textrm{\Pone}}]
  For every $(\xt_0,\yt_0)\in\R^2$ find a pair trajectory-control joining $(\xt_0,\yt_0)$ to $\Xzerot=\left(0,-\rmin\right)$, which is time-optimal for the control system
  \begin{equation}\label{eq:dubinstilde}
    \left\{
    \begin{aligned}
      \dot{\xt} &= v+u\yt\\
      \dot{\yt} &= -u\xt
    \end{aligned}
    \right.,
    \quad (u,v)\in\controlset.
  \end{equation}
\end{itemize}

\begin{remark}
One can read, in equation (\ref{eq:dubinstilde})
that $(0,-\rmin)$ is an equilibrium point
corresponding to the control values $u=\umax$ and $v=\vmin$.
\end{remark}

The family of all solutions to problem \textrm{\Pone} for $(\xt_0,\yt_0)\in\R^2$ is called the \textit{time-optimal synthesis}.

Following a standard approach (see \cite{BoscainPiccoli2004_BOOK}) for time-optimal control syntheses, it is convenient to rephrase problem {\textrm{\Pone}} as an equivalent problem
backward in time.
Hence, changing the sign of the dynamics, the following equivalent time-optimal problem is considered:
\begin{itemize}
  \item[{\textrm{\Ptwo}}]
  For every $\left(\xt_{f},\yt_{f}\right)\in\R^2$ find a pair trajectory-control joining $\Xzerot=\left(0,-\rmin\right)$ to $\left(\xt_{f},\yt_{f}\right)$, which is  time-optimal for the control system
  \begin{equation}\label{eq:system-reduced}
    \left\{
    \begin{aligned}
    \dot{\xt} &= -v-u\yt\\
    \dot{\yt} &= u\xt
    \end{aligned}
    \right.,
    \quad (u,v)\in\controlset.
  \end{equation}
\end{itemize}
Once problem \textrm{\Ptwo} is solved, then the time-optimal
synthesis (corresponding to problem \textrm{\Pone}) is obtained straightforwardly following the travelled trajectories backward.
\begin{remark}\label{rem:reparametrization}
Note that
up to a dilation in the $(x,y)$-plane
and a dilation of time (a time-reparametrization with constant derivative),
we may assume that
$[-\umax,\umax]\times[\vmin,\vmax]=[-1,1]\times[1,\p]$ ($\p=\vmax/\vmin$).
This normalization is used to simplify the treatment in Sections \ref{sec:applicationI} and \ref{sec:applicationII}.
\end{remark}

%
%
\section{Time-optimal synthesis on \texorpdfstring{$\R^2$}{}}\label{sec:theory}

In this section, following the same ideas as those developed by Boscain, Bressan, Piccoli and Sussmann
in \cite{BoscainPiccoli2001,BoscainPiccoli2004_BOOK,BressanPiccoli1998,Piccoli1996,Sussmann1987_REG-SYN}
for optimal syntheses on two-dimensional manifolds for single input control systems,
we introduce important definitions and develop basic facts about optimal syntheses on $\R^2$
for control-affine systems 
with two bounded controls
(which are different from those studied \startmodif
in \cite{BoscainChambrionCharlot2005}
\stopmodif and \cite{Boscain2014b}).
This part is widely inspired by the book \cite{BoscainPiccoli2004_BOOK} and extends some of its results.

\startmodif
The definitions and results given in Subsections \ref{sec:PMP}, \ref{sec:basic-def} and \ref{sec:abnormal-traj} are valid in $\R^n$ for any $n$. However, starting from Subsection \ref{sec:singular-traj} results are valid and make sense only for $n=2$.
\stopmodif
%
%
\subsection{Pontryagin Maximum Principle}\label{sec:PMP}

Let $\FF$ and $\GG$ be two smooth and complete vector fields on $\R^n$.
Define the control variable $\UU=(u,v)$ and the control set $\controlset=[-\umax,\umax]\times[\vmin,\vmax]\subset\R^2$, where $\umax$, $\vmin$, $\vmax$ are assumed to be positive.
Consider the following general control-affine time-optimal problem.
\begin{itemize}
  \item[{\textrm{\Pgen}}]
  For every $\qzero$ and $\qfinal$ in $\R^n$ find the pair trajectory-control joining $\qzero$ to $\qfinal$, which is  time-optimal for the control system
  \begin{equation}\label{eq:general-system}
    \dot \qq = v\FF(\qq) + u \GG(\qq),\quad
    \qq \in \R^n,\quad
    (u,v) \in \controlset.
  \end{equation}
\end{itemize}
\begin{definition}[admissible control/trajectory]
An \textit{admissible control} for system (\ref{eq:general-system}) is an
essentially bounded function $\UU(\cdot): [\Interval]\to \controlset$.
An \textit{admissible trajectory} is a solution to (\ref{eq:general-system})
such that
$\dot
\qq(t)=v(t)\FF(\qq(t))+u(t)\GG(\qq(t))$
a.e. on $[\Interval]$ for some  admissible control $\UU(\cdot)$. 
\end{definition}
Thanks to the compactness of the set of controls, the convexity of the set of velocities,
and the completeness of the vector fields,
Filippov’s theorem (see, for instance, \cite{Agrachev-Sachkov-2004}) yields:
\begin{proposition}
  For any pair of points in $\R^n$, there exists a time-optimal trajectory joining them.
\end{proposition}

The main tool to compute time-optimal trajectories is the Pontryagin Maximum Principle (PMP).
A general version of PMP can be found in \cite{Agrachev-Sachkov-2004}.
The following theorem is a version of PMP for control systems of the form (\ref{eq:general-system})
that we state in our own context only.

\begin{theorem}[PMP]\label{thm:PMP}
  Consider the control system (\ref{eq:general-system}).
  For every $(\pp,\qq,\UU)\in \R^n\times\R^n\times \controlset$, define the Hamiltonian function
  \begin{equation}
    H(\pp,\qq,\UU)=v\left\langle \pp,F(\qq)\right\rangle +u\left\langle \pp,G(\qq)\right\rangle .\label{eq:hamiltonian}
  \end{equation}
  Let $\UU(\cdot)$ be an admissible time-optimal control defined on $[\Interval]$ and let $\qq(\cdot)$ be the corresponding trajectory.
  Then there exist a never vanishing Lipschitz covector (or adjoint vector) $\pp(\cdot):t\in[\Interval]\mapsto \pp(t)\in \R^n$ and a non negative constant $\lambda$ such that for almost all $t\in[\Interval]$:
  \begin{enumerate}[\rm i.]
    \item\label{PMP0}
      $\dot \qq(t)
      =\displaystyle{\frac{\partial H}{\partial \pp}(\pp(t),\qq(t),\UU(t))}$,
      \vspace{0,2em}
    \item\label{PMP1}
      $\dot \pp(t)
      =-\displaystyle{\frac{\partial H}{\partial \qq}(\pp(t),\qq(t),\UU(t))}$,
      \vspace{0,2em}
    \item\label{PMP2}
      $H(\pp(t),\qq(t),\UU(t))\displaystyle=\max_{W\in\,\controlset}H(\pp(t),\qq(t),W)$,
      \vspace{0,2em}
    \item\label{PMP4}
      $H(\pp(t),\qq(t),\UU(t))=\lambda\ge0$.
      \vspace{0,2em}
  \end{enumerate}
\end{theorem}

A pair trajectory-control $(\qq(\cdot),\UU(\cdot))$ (resp. a triplet $(\pp(\cdot),\qq(\cdot),\UU(\cdot))$) satisfying the conditions given by the PMP is said to be an \textit{extremal trajectory} (resp.  an \textit{extremal}).
An extremal  corresponding to $\pzero=0$ is said to be \textit{abnormal}, otherwise we call it \textit{normal}.

\begin{remark}\label{rem:h-cont}
Notice that,
up to change $\UU(\cdot)$ on a set of measure zero,
an extremal control can always be chosen so that the function
$t\mapsto H(\pp(t),\qq(t),\UU(t))$ is continuous.
Consequently,
we may always assume
(without loss of generality)
that condition \ref{PMP4} of PMP is valid everywhere.
\end{remark}

%
%
\subsection{Basic definitions}\label{sec:basic-def}

\begin{definition}[Switching functions]
  Let $\qq(\cdot)$ be an extremal trajectory.
  The corresponding $u$- and $v$-switching functions
  are (the differentiable functions) 
  defined respectively as
  \[
    \phiu(t)=\left\langle \pp(t),G\left(\qq(t)\right)\right\rangle 
    \textrm{~and~}
    \phiv(t)=\left\langle \pp(t),F\left(\qq(t)\right)\right\rangle.
  \]
\end{definition}
Switching functions are very important since their analysis determine when the corresponding control may change.
Unlike in the single input case, the switching functions are differentiable but not necessarily $C^1$.

In the following three definitions (\startmodif
Definitions
\stopmodif \ref{def:bang}-\ref{def:switchingtime}), $\qq(\cdot)$ is an extremal trajectory defined on the time interval $[\Interval]$ and $\UU(\cdot):[\Interval]\to\controlset$ is the corresponding control.
\begin{definition}[Bang]\label{def:bang}
  $\UU(\cdot)$ is said to be a $u$-\textit{bang} (resp. $v$-\textit{bang}) control if, for a.e. $t\in[\Interval]$, $u(t)=-\umax$ (or $u(t)=\umax$) (resp. $v(t)=\vmin$ (or $v(t)=\vmax$)).
  $\UU(\cdot)$ is a \textit{bang} control if, for a.e. $t$ in $[\Interval]$, it is $u$-\textit{bang} and $v$-\textit{bang}. 
  A finite concatenation of bang controls is called a \textit{bang-bang} control.
\end{definition}

\begin{definition}[Singular]\label{def:singular}
  We say that $\UU(\cdot)$ is a \emph{$u$-singular} control (resp. \emph{$v$-singular}) if the corresponding switching function $\phiu$ (resp. $\phiv$) vanishes identically on $[\Interval]$. If $\phiu$ and $\phiv$ both vanish identically on $[\Interval]$, we say that $\UU(\cdot)$ is \emph{totally singular}.
\end{definition}

\begin{definition}[Switching times]\label{def:switchingtime}
A \textit{$u$-switching time} of $\UU(\cdot)$
is a time $\tsw\in (t_0,t_1)$ such that,
for a sufficiently small $\varepsilon>0$,
$u(t)=\umax$ for a.e. $t \in (\tsw-\epsilon,\tsw]$ and
$u(t)=-\umax$ for a.e. $t\in(\tsw,\tsw+\epsilon]$
or vice-versa.
A $v$-switching time is defined similarly.
A $\uv$-switching time is a time
that is both a $u$- and a $v$-switching time.
If $\tsw$ is a switching time,
the corresponding point $\qq(\tsw)$ on 
the trajectory $\qq(\cdot)$ is called a \textit{switching point}.
\end{definition}

%
%
\subsection{Abnormal trajectories}\label{sec:abnormal-traj}

The following lemma gives some information on the nature of abnormal extremals of system (\ref{eq:general-system}).

\begin{lemma}\label{lem:Abnormals_form}
Let $\gamma(\cdot)=\left(\pp(\cdot),\qq(\cdot),\UU(\cdot)\right)$ be an abnormal extremal defined on $[\Interval]$.
Assume that $\gamma(\cdot)$ is never totally singular in restriction to any subinterval of $[\Interval]$.
Then $\gamma(\cdot)$ is bang with $v(\cdot)=\vmin$ on an open dense subset of $[\Interval]$.
\end{lemma}

\begin{proof}
  Since the considered extremal is abnormal, we have, for all $t\in[\Interval]$, 
  \begin{equation}\label{eq:abnormal-condition}
    H\left(\pp(t),\qq(t),\UU(t)\right)=v(t)\phiv(t) + \umax\left|\phiu(t)\right|=0.
  \end{equation}
  From (\ref{eq:abnormal-condition}), we conclude that a zero value of $\phiu(\cdot)$ implies a zero value of $\phiv(\cdot)$ and conversely.
  Hence, the extremal being never totally singular in restriction to any subinterval of $[\Interval]$, $\phiu(\cdot)$ and $\phiv(\cdot)$ cannot vanish identically on any subinterval of $[\Interval]$.
  Consequently,
  the set $E = \{ t\in[\Interval]\mid \phiv(t)\neq0 \}$ is open (since $\phiv$ is continuous) and dense in $[\Interval]$.
  From (\ref{eq:abnormal-condition}) again, we necessarily have $\phiv(t) < 0$ for all $t\in E$.
  Consequently, the considered extremal is bang (on $E$) with $v(\cdot)=\vmin$.
\end{proof}

\startmodif
\begin{remark}
According to the forthcoming Remark \ref{rem:tot-sing},
generically, on $\R^2$, an abnormal is not totally singular.
\end{remark}
\stopmodif


%
%
\subsection{Singular trajectories}\label{sec:singular-traj}

\startmodif
From now and until the end of the paper, we assume $n=2$, i.e. $\qq\in\R^2$.
\stopmodif
Let us introduce the functions\footnote{If $F_1$ and $F_2$ are two vector fields, $[F_1,F_2]$ denotes their Lie bracket.}
\begin{eqnarray*}
\DA(\qq) & = & \det\left(F(\qq),G(\qq)\right),\\
\DBu(\qq) & = & \det\left(G(\qq),[F,G](\qq)\right),\\
\DBv(\qq) & = & \det\left(F(\qq),[F,G](\qq)\right),
\end{eqnarray*}
whose zero sets are \textit{fundamental loci} (see \cite{BoscainPiccoli2004_BOOK}) in the construction of the optimal synthesis.

\begin{remark}
Notice that,
although the functions $\DA$, $\DBu$ and $\DBv$ depend on coordinates in \startmodif $\R^2$ \stopmodif,
the sets $\DA^{-1}(0)$, $\DBu^{-1}(0)$ and $\DBv^{-1}(0)$ do not;
indeed, they are intrinsic objects related with the control system (\ref{eq:general-system}).
\end{remark}
The following lemma which is a direct generalization of \cite[Theorem 12 page 47]{BoscainPiccoli2004_BOOK} is stated without proof.
\begin{lemma}\label{lem:singular-trajectories}
  $u$-singular (resp. $v$-singular) trajectories are contained in the set $\DBu^{-1}(0)$ (resp. $\DBv^{-1}(0)$).
\end{lemma}
\begin{lemma}[$\uv$-singular trajectories]\label{lem:uv-sing}
  $\uv$-singular trajectories are contained in the set
  $\DA^{-1}(0) \bigcap\DBu^{-1}(0) \bigcap\DBv^{-1}(0)$.
\end{lemma}
\begin{proof}
  Let $\left(\pp(t),\qq(t),\UU(t)\right)$ be a $\uv$-singular extremal defined on $[\Interval]$.
  Then for a.e. $t\in[\Interval]$,
  \begin{eqnarray*}
    \phiv(t) & = & \left\langle \pp(t),F\left(\qq(t)\right)\right\rangle =0,\\
    \phiu(t) & = & \left\langle \pp(t),G\left(\qq(t)\right)\right\rangle =0,
  \end{eqnarray*}
  which implies that $\pp(t)$ is orthogonal to both $F(\qq(t))$ and $G(\qq(t))$.
  But, according to PMP, $\pp(t)$ cannot vanish,
  hence,
  $F$ and $G$ must be parallel along $\qq(\cdot)$.
  We thus get $\qq(\cdot) \subset \DA^{-1}(0)$.
  Moreover,
  according to Lemma \ref{lem:singular-trajectories},
  we also have \(\DBu^{-1}(0) \bigcap\DBv^{-1}(0)\).
\end{proof}
\begin{remark}\label{rem:tot-sing}
Although it is not addressed here,
it can be proved that
the intersection $\DA^{-1}(0) \bigcap\DBu^{-1}(0) \bigcap\DBv^{-1}(0)$ is generically empty.
In other words, generically, on $\R^2$, there is no totally singular trajectories.
\end{remark}
The next lemma describes the kind of switches that may occur  along singular arcs.
\begin{lemma}\label{lem:Singular_control_and_switch}
  Along a $u$-singular trajectory which is not totally singular,
  $v$ is a.e. equal to $\vmax$.
  Along a $v$-singular trajectory which is not totally singular, a $u$-switching cannot occur.
\end{lemma}
\begin{proof}
  Let $(\qq(\cdot),\UU(\cdot))$ be a $u$-singular extremal trajectory on $[\Interval]$
   which is not totally singular.
   According to the PMP and since the trajectory is not $v$-singular,
   we have
  \[
  H\left(\pp(t),\qq(t),\UU(t)\right)
  =
  v(t)\phiv(t)
  =\lambda
  >0
  \]
  for a.e. $t\in[\Interval]$.
  Hence, $\phiv(\cdot)$ is positive on $[\Interval]$ and consequently
  $v(t)=\vmax$ for a.e. $t\in[\Interval]$.
  The proof for a $v$-singular trajectory is similar but the PMP yields not the value of the control $u$ along it.
\end{proof}


%
%
\subsection{Switchings}

\subsubsection{$u$-switchings}

\begin{lemma}
  Along a normal extremal trajectory, a $u$-switching can occur if and only if $v=\vmax$.
\end{lemma}
\begin{proof}
  Let $\qq(\cdot)$ be a normal trajectory on $[\Interval]$ and let $\UU(\cdot)$ be the corresponding control. Let $\tau\in(t_0,t_1)$ be a $u$-switching time, i.e. $\phiu(\tau)=0$. From the PMP, we have 
  \[H\left(\pp(\tau),\qq(\tau),\UU(\tau)\right)=v(\tau)\phiv(\tau)>0,\]
  with $v(\tau)>0$.
  Hence,
  $\phiv(\tau)>0$ and thus for $\epsilon$ sufficiently small,
  $v(t)=\vmax$ for a.e. $t\in[\tau-\epsilon,\tau+\epsilon]$.
\end{proof}

\begin{remark}\label{rem:2-v-switch-in-a-row}
  The previous lemma implies in particular that, along a normal trajectory, a $v$-switching from $\vmax$ to $\vmin$ is necessarily followed by another $v$-switching from $\vmin$ to $\vmax$ before a $u$-switching occurs.
\end{remark}

\subsubsection{$\uv$-switchings} 

\begin{lemma}[$\uv$-switchings]\label{lem:u-v-switchings_cannot_occur} 
A $\uv$-switching cannot occur along an extremal trajectory.
\end{lemma}
\begin{proof}
Let $\left(\pp(\cdot),\qq(\cdot),\UU(\cdot)\right)$ be an extremal
and let $\tau$ be a $\uv$-switching time.
Since the Hamiltonian is constant along trajectories with
  \[
  H\left(\pp(\tau),\qq(\tau),\UU(\tau)\right)
  =  
  v(\tau)\phiv(\tau)+\umax\left|\phiu(\tau)\right|=0,
  \]
the extremal is abnormal.
But,
according to Lemma \ref{lem:Abnormals_form},
an abnormal extremal admits only $(u(\cdot),v(\cdot))=(\pm\umax,\vmin)$ as bang controls
and thus cannot have a $\uv$-switching.
\end{proof}

%
%
\subsection{Special domains}

Using the sets
$\DA^{-1}\left(0\right)$, $\DBu^{-1}\left(0\right)$ and $\DBv^{-1}\left(0\right)$
defined in Section \ref{sec:singular-traj},
we can define domains in which a control can switch at most once.
Consider a point $\qq\notin\DA^{-1}(0)$.
Then $F(\qq)$ and $G(\qq)$ are linearly independent and form a basis.
An easy computation shows that
\begin{equation}\label{eq:fg}
  \left[F,G\right](\qq)
  =f(\qq)F(\qq)+g(\qq)G(\qq),
\end{equation}
with
\[
  f(\qq)=-\frac{\DBu(\qq)}{\DA(\qq)} \qquad \text{and} \qquad 
  g(\qq)=\frac{\DBv(\qq)}{\DA(\qq)}.
\]
\begin{lemma}\label{lem:one-switch-per-cadrant}
A normal and non-singular trajectory
along which $f>0$ (resp. $f<0$)
admits at most one $u$-switching
and
necessarily from $-\umax$ to $\umax$ (resp. from $\umax$ to $-\umax$).
Similarly,
a normal and non-singular trajectory
along which $g>0$ (resp. $g<0$)
admits at most one $v$-switching
and necessarily from $\vmax$ to $\vmin$
(resp. from $\vmin$ to $\vmax$).
\end{lemma}
\begin{proof}
Let $\left(\pp(\cdot),\qq(\cdot),\UU(\cdot)\right)$ be a normal,
non-singular extremal defined on $[\Interval]$
along which $f>0$.
At any $\tau\in[\Interval]$ such that $\phiu(\tau)=0$,
we have
\[
  H\left(\pp(\tau),\qq(\tau),\UU(\tau)\right)
  = v(\tau)\phiv(\tau)
  =\lambda>0,
\]
which, with equation (\ref{eq:fg}), implies
\begin{eqnarray}
\dot{\phiu}(\tau)
&=& v(\tau)\lang \pp(\tau),[F,G](\qq(\tau))\rang \nonumber\\
&=& v(\tau)f(\qq(\tau))\phiv(\tau)+v(\tau)g(\qq(\tau))\phiu(\tau) \nonumber\\
&=& v(\tau)\phiv(\tau)f(\qq(\tau)) \label{eq:dotphiu-f}
>0.
\end{eqnarray}
Since,
$\phiu$ is continuous, we conclude that $\phiu$ can vanish at most once in $[\Interval]$.
Moreover,
according to inequality (\ref{eq:dotphiu-f}),
$\dot{\phiu}(\tau)$ and $f(\qq(\tau))$ have same sign.
Consequently,
if $\phiu(\tau)=0$,
then $\dot \phiu(\tau)>0$ and
$u(\cdot)$ switches at $\tau$ from $-\umax$ to $\umax$.
Reasoning similarly with $\phiu$ replaced by $\phiv$ gives the second part of the lemma.
\end{proof}

%
%
\section{Construction of time-optimal synthesis for the reduced system}\label{sec:applicationI}

In this section,
we apply the results obtained in the previous sections to solve problem \textrm{\Ptwo}.
\startmodif
 Although the problem
 \stopmodif is similar to the one studied
in \cite{MaillotBoscainGauthierSerres2015},
the resolution is much more complicated due to the presence of a second control.

In this section, for the sake of clarity and without loss of generality,
we assume,
according to Remark \ref{rem:reparametrization}, that
$\controlset=[-1,1]\times[\vm,\VM]$.
Note moreover that, in this case, $\rmin=\vm$.

%
%
\subsection{Application of the Pontryagin Maximum Principle}

First of all,
notice that system (\ref{eq:system-reduced}) is of the form (\ref{eq:general-system}) with
\[
  \qqt        = \begin{pmatrix}   \xt \\ \yt  \end{pmatrix},\quad
  F(\qqt)     = \begin{pmatrix}  -1   \\ 0    \end{pmatrix} \textrm{ and}\quad
  G(\qqt)     = \begin{pmatrix}  -\yt \\ \xt  \end{pmatrix}.
\]
We apply PMP to \textrm{\Ptwo}.
The control-dependent Hamiltonian function of PMP is
\[
  H\big(\qqt,\pp,\UU\big) = -v\px+u\left(\py \xt-\px \yt\right),
\]
with $\pp = \left(\px,\py\right)\in\R^2$ being the covector.
The adjoint system is thus given by 
\begin{equation}\label{eq:adjoint_system}
  \left\{
  \begin{aligned}
    \dot{\px}(t) &= -u(t)\py(t)\\
    \dot{\py}(t) &=  u(t)\px(t)
  \end{aligned}.
  \right.
\end{equation}
and the switching functions are
\begin{eqnarray}
  \phiu(t) & = & \py(t)\xt(t)-\px(t)\yt(t), \label{eq:switching_function_u}\\
  \phiv(t) & = & -\px(t). \label{eq:switching_function_v}
\end{eqnarray}
The maximality condition of the PMP reads
\[
  H\big(\qqt(t),\pp(t),\UU(t)\big)
  =\max_{\left(u,v\right)\in\,\controlset}\big(u\phiv(t)+v\phiu(t)\big)
  =\lambda,
\]
and yields the controls
\begin{equation}\label{E:bang-controls}
  u(t)=
  \begin{cases}
    -1  &\textrm{if~} \phiu(t)<0\\
    \phantom{-} 1   &\textrm{if~} \phiu(t)>0
  \end{cases},
  \quad
  v(t)=
  \begin{cases}
    \vm &\textrm{if~} \phiv(t)<0\\
    \VM &\textrm{if~} \phiv(t)>0
  \end{cases}.
\end{equation}
\begin{remark}
  The cases where the switching functions vanish identically is addressed in the next subsection.
\end{remark}

%
%
\subsection{Singular trajectories}

Let us compute the quantities
\[
  \DA(\xt,\yt) = -\xt ,\quad \DBu(\xt,\yt)=\yt \quad \DBv(\xt,\yt)=1,
\]
\vspace{-1em}
\[
  f(\xt,\yt) = \frac{\yt}{\xt} ,\quad g(\xt,\yt)=-\frac{1}{\xt}.
\]
Lemmas \ref{lem:singular-trajectories},
\ref{lem:uv-sing}
and \ref{lem:Singular_control_and_switch}
imply that

\begin{itemize}
\item
there exists no $v$-singular trajectory
(and consequently no totally singular trajectory)
since $\DBv^{-1}(0)=\emptyset$;

\item
$u$-singular trajectories are contained in the
set\footnote{According to \cite{BoscainPiccoli2004_BOOK}, the set $\{(\tilde x,\tilde y)\in\R^2 \mid \tilde y=0\}$ is a turnpike.}
\[
\DBu^{-1}(0)=\left\{ \left(\xt,\yt\right)\in\mathbb{R}^{2} \mid \yt=0\right\}.
\]
To compute the corresponding control, 
we differentiate w.r.t. $t$ the function $\phiu$
(which is identically zero).
A straightforward calculation yields 
\[
  \begin{aligned}
    \dot \phiu(t) & = \VM\lang \pp(t),\left[F,G\right](\qq(t))\rang
                  = -\VM\py(t)
                    = 0,\\
    \ddot \phiu(t) & = \VM u(t)\lang \pp(t),\left[G,\left[F,G\right]\right](\qq(t))\rang 
                      =\VM u(t)\px(t)
                      =0,
      \end{aligned}
\]
which,
taking into account that $(\px(t),\py(t))$ never vanishes, implies that,
along $u$-singular trajectories,
$u(\cdot)$ vanishes identically.
Note that this is quite intuitive since $u=0$ is
the only control that allows the trajectory to stay on the $x$-axis.
\end{itemize}

%
%
\subsection{Notation}\label{sec:notation}

The following table defines a naming convention for the five possible optimal controls in order to simplify the description of the trajectories.
\begin{table}[h]
  \begin{centering}
    \begin{tabular}{|c|c|}
    \hline Control          & Notation     \tabularnewline
    \hline $(-1,\vm)$     & $\cm$          \tabularnewline
    \hline $(1,\vm)$      & $\cp$          \tabularnewline
    \hline $(-1,\VM)$     & $\cM$          \tabularnewline
    \hline $(1,\VM)$      & $\cP$          \tabularnewline
    \hline {$u$-singular}   & $\cs$          \tabularnewline
    \hline 
    \end{tabular}
    \caption{Notation of the five possible optimal controls\label{tab:Control-notation}}
  \par\end{centering}
\end{table}
The letter stands for $u$ equals to plus $(\cp,\cP)$ or minus $(\cm,\cM)$ one or to the singular control.
The lower case and upper case correspond to $\vm$ and $\VM$ respectively.
We define $\traj{i}$ as the trajectory starting from $\Xzerot$ with the control 
$i\in\{\cm,\cM,\cp,\cP,\cs\}$.
Similarly,
$\traj{ij}$ denotes the trajectory switching from $\traj{i}$ at time $\Tsw{ij}$ with the control
$j\in\{\cm,\cM,\cp,\cP,\cs\}\setminus\{i\}$,
and $\zphiu{ij}$ (resp. $\zphiv{ij}$) denotes a zero of the switching function on $u$ (resp. $v$) corresponding to the trajectory $\traj{ij}$.

%
%
\subsection{Optimal synthesis algorithm}\label{sec:algorithm}

Since $\Xzerot$ is an equilibrium point of system \eqref{eq:dubins} for the control
\((1,\vm)\) and since 
$\Xzerot\notin\DBu^{-1}(0)$,
there are, a priori, three possible optimal starting trajectories candidates:
$\traj{\cM}$, $\traj{\cm}$ and $\traj{\cP}$ corresponding to the bang controls  \( \left(-1,\VM\right) \), \( \left(-1,\vm\right) \) and \( \left(1,\VM\right) \) respectively (see Fig. \ref{fig:Integral-curves-q0}).
\begin{figure}[!ht]
  \centering
  \def\svgwidth{0.7\columnwidth}
  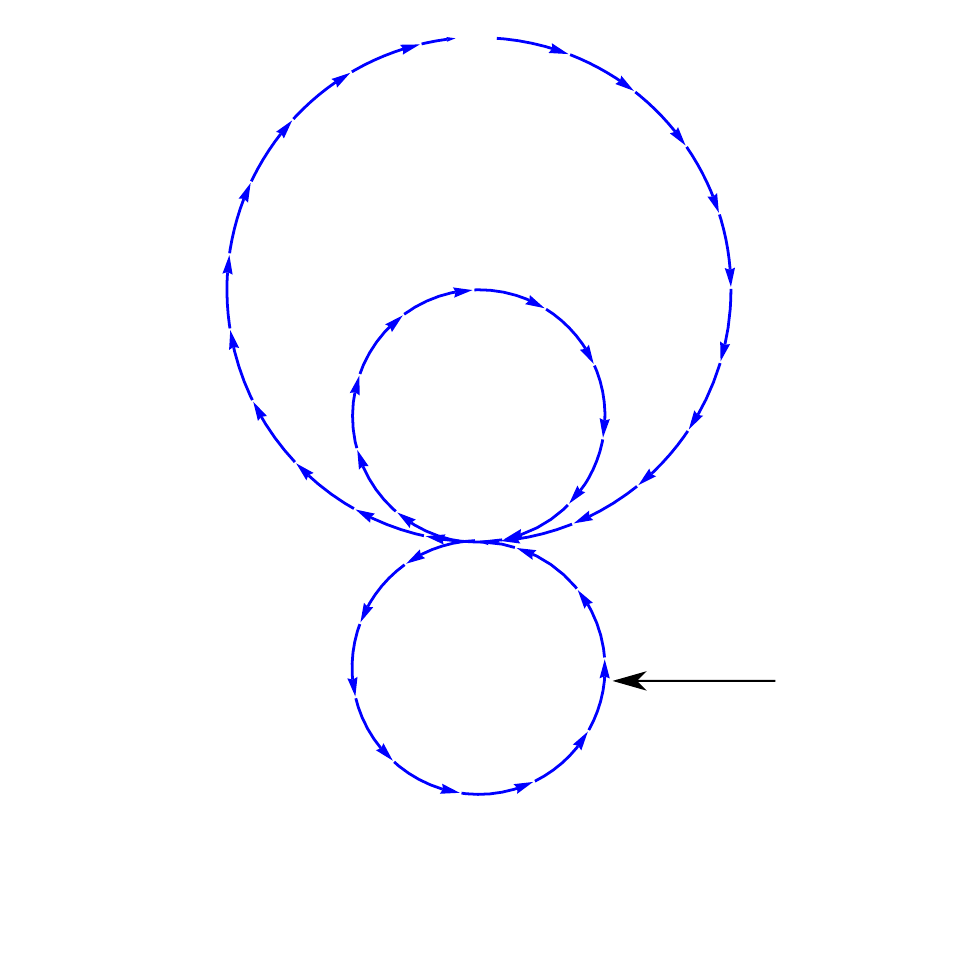
  \caption{Candidate extremal trajectories of problem \textrm{\Ptwo} issued from $\Xzerot$\label{fig:Integral-curves-q0}}
\end{figure}
The time-optimal synthesis is constructed following the three steps described below.
\begin{enumerate}[Step 1]
  \item For each bang trajectory starting from $\Xzerot$, compute the last time at which the trajectory is extremal (or has lost its optimality by intersecting itself) and study which kind of extremal trajectories can bifurcate from it.
  \item For each bang or singular trajectory bifurcating from one of the starting trajectories, compute the last time at which it is extremal.
  If there are intersections among trajectories, we cancel those parts that are not optimal (among trajectories already computed up to this step).
  \item For each trajectory
  computed at the previous step that did not loose its optimality, 
  prolong it with the next bang or singular trajectory up to the last time at which it is extremal.
  If there are intersections among trajectories, cancel those parts that are not optimal (among trajectories already computed up to this step).
\end{enumerate}
Then the synthesis is built recursively, repeating Step 3 until no new trajectories are generated.
In our case, four applications of Step 3 are necessary.

\begin{remark}
  Notice that, although Step 2 and Step 3 seem similar, Step 2 does not add any trajectory to the synthesis whereas Step 3 does. Step 1 and 2 correspond to an initialisation of the algorithm whereas Step 3 is the step that construct the synthesis recursively. 
\end{remark}

%
%
\subsection{Expression of the adjoint vector}
Since the starting control $(u_0,v_0)=(u(0),v(0)) \in \{(-1,\vm),(-1,\VM),(1,\VM)\}$,
the solution to system (\ref{eq:adjoint_system})
with the normalization $|\pp(0)|=1$ is
\begin{equation}\label{eq:sol_adjoint}
\left\{
    \begin{aligned}
      \px(t) &=\cos\left(\alpha+t\right)\\
      \py(t) &=u_0\sin\left(\alpha+t\right), \quad u_0\in\{-1,1\},
    \end{aligned}
    \right.
\end{equation}
where $\alpha$ is defined by $\pp(0)=(\cos\alpha,u_0\sin\alpha)$.
The condition
\ref{PMP4} of PMP written at the initial point implies that
\[
  \left(-v_0+u_0\right)\px(0)
  \ge0.
\]
Since 
\(
\left(u_0,v_0\right)\neq\left(1,\vm\right),
\)
\(
\left(-v_0+u_0\right)
\)
\startmodif
is non-positive.
\stopmodif
Hence,
\(
\px(0)=\cos\alpha\le0
\)
and
$\alpha\in[\frac{\pi}{2},\frac{3\pi}{2}]$.
The following study consists in analyzing the behaviour of the extremal trajectories depending on the value of
$(\alpha,u_0,v_0) \in [\frac{\pi}{2},\frac{3\pi}{2}]\times\{(-1,\vm),(-1,\VM),(1,\VM)\}$.

%
%
\subsection{Step 1: Analysis of the trajectories starting from \texorpdfstring{$\Xzerot$}{}}\label{S:Step1}
This section details the first step (and only this one) of the algorithm described in Section \ref{sec:algorithm}.

\subsubsection{Trajectories corresponding to control ${\left(-1,\VM\right)}$}\label{sec:(-1,eta)}

The trajectory $\traj{\cM}$ starting from $\Xzerot$ with the control \( \left(-1,\VM\right) \)
has coordinates
\[
  \left\{
  \begin{aligned}
    \xt(t) &=-\left(\VM+\vm\right)\sin t\\
    \yt(t) &=\VM-\left(\VM+\vm\right)\cos t,
  \end{aligned}
  \right.
\]
and from (\ref{eq:sol_adjoint}), the coordinates of the adjoint vector are
\[
  \left\{
  \begin{aligned}
    \px(t)&=\cos\left(t+\alpha\right)\\
    \py(t)&=-\sin\left(t+\alpha\right),
  \end{aligned}
  \right.
\]
with $\alpha\in[\frac{\pi}{2},\frac{3\pi}{2}]$. 
It follows that the switching functions are 
\begin{eqnarray*}
  \phiu(t) &=& \left(\VM+\vm\right)\cos\alpha-\VM\cos\left(t+\alpha\right),\\
  \phiv(t) &=& -\px(t)=-\cos\left(t+\alpha\right).
\end{eqnarray*}
Recall that the sign of each of these functions determines the value of the corresponding control. 
It is thus fundamental to study when does a switching function change its sign.
Depending on the initial value of the covector parameterized by $\alpha$ the following cases have to be distinguished.
\begin{itemize}
  \item For $\alpha={\pi}/{2}$, $\phiu$ starts with value zero and then takes positive values: from \eqref{E:bang-controls}, it cannot correspond to	a trajectory starting with control ${\left(-1,\VM\right)}$.

  \item For $\asing=\arccos\left(-{\VM}/({\VM+\vm})\right)$,
  		$\phiu$ starts with negative values and has a zero of order two at 
  		$\tsing=\pi-\asing$.
  		$\phiv$ starts with positive values and vanishes after $\phiu$ does.
      \startmodif
      The trajectory reaches the $\tilde x$-axis at time $\tsing$. In this case, at time $t=\tsing$, either the trajectory becomes $u$-singular, switches to ${\left(1,\VM\right)}$ or does not switch. In the latter case, there is a $v$-switching at time
      \(
        \Tsw{\cM\cm}(\alpha)=-\alpha+{3\pi}/{2}
      \).
      \stopmodif
  
  \item For every $\alpha\in\left({\pi}/{2},\asing\right)$,
	  	$\phiu$ starts with negative values and $\phiv$ with positive values.
	  	The function that vanishes first is $\phiu$. At time
	 	\(
	    	\Tsw{\cM\cP}(\alpha) = -\alpha+\arccos\left(\cos\alpha({\VM+\vm})/{\VM}\right)
	 	\),
	 	$\phiu$ changes its sign and a $u$-switching occurs.

  \item For every $\alpha\in\left(\asing,{3\pi}/{2}\right)$,
  		$\phiu$ is always negative. $\phiv$ starts with positive values and 
  		changes sign at time
  		\(
  			\Tsw{\cM\cm}(\alpha)=-\alpha+{3\pi}/{2}
  		\).
  		There is a $v$-switching at this time.
  
  \item For $\alpha={3\pi}/{2}$, $\phiv$ starts with value zero and then takes negative values: from \eqref{E:bang-controls}, it cannot correspond to
  		a trajectory starting with control ${\left(-1,\VM\right)}$.
\end{itemize}
\startmodif
From this analysis, we define the following families of trajectories issued from $\traj{M}$:
\begin{description}
  \item[Family 1] trajectories $\traj{Ms}$ corresponding to 
    $\alpha=\asing=\arccos\left(-{\VM}/({\VM+\vm})\right)$.
  \item [Family 2] trajectories $\traj{MP}$ corresponding to 
    $\alpha\in\left({\pi}/{2},\asing\right]$.
  \item [Family 3] trajectories $\traj{Mm}$ corresponding to 
    $\alpha\in\left[\asing,{3\pi}/{2}\right)$.
\end{description}
\stopmodif


\subsubsection{Trajectories corresponding to control $\left(1,\VM\right)$}

The trajectory $\traj{P}$ starting from $\Xzerot$ with the control \( \left(1,\VM\right) \)
has coordinates
\[
\left\{
\begin{aligned}
\xt(t) &=-\left(\VM-\vm\right)\sin t\\
\yt(t) &=-\VM+\left(\VM-\vm\right)\cos t
\end{aligned}
\right.
\]
and from (\ref{eq:sol_adjoint}), the coordinates of the corresponding adjoint vector are
\[
\left\{
\begin{aligned}
\px(t) &=\cos\left(t+\alpha\right)\\
\py(t) &=\sin\left(t+\alpha\right),
\end{aligned}
\right.
\]
with $\alpha\in\left[\frac{\pi}{2},\frac{3\pi}{2}\right]$.
It follows that the switching functions are 
\begin{eqnarray*}
	\phi_{u}(t) & = & -\left(\VM-\vm\right)\cos\alpha+\VM\cos\left(t+\alpha\right),\\
	\phi_{v}(t) & = & -\cos\left(t+\alpha\right).
\end{eqnarray*}
Similarly to Section \ref{sec:(-1,eta)}, depending on the value of $\alpha$ the following analysis may be performed.
\begin{itemize}
	\item For $\alpha\in [\frac{\pi}{2},\frac{3\pi}{2})$, $\phi_{u}$
		starts with a nonpositive value and then takes negative values: from \eqref{E:bang-controls},
		it cannot correspond to	a trajectory starting with control ${\left(1,\VM\right)}$.

	\item For $\alpha=\frac{3\pi}{2}$, $\phi_{v}$ starts with value zero and then takes
		negative values: from \eqref{E:bang-controls}, it cannot correspond to	a trajectory starting with control ${\left(1,\VM\right)}$.
\end{itemize}
We conclude that there is no optimal trajectory starting from $\Xzerot$ 
with the controls \( \left(1,\VM\right) \).

\subsubsection{Trajectories corresponding to control $\left(-1,\vm\right)$} 

The trajectory $\traj{m}$ starting from $\Xzerot$ with the control \( \left(-1,\vm\right) \)
has coordinates
\[
  \left\{
  \begin{aligned}
    \xt(t) &=-2\sin t\\
    \yt(t) &=\vm-2\cos t,
  \end{aligned}
  \right.
\]
and from (\ref{eq:sol_adjoint}), the coordinates of the adjoint vector are
\[
  \left\{
  \begin{aligned}
    \px(t)&=\cos\left(t+\alpha\right)\\
    \py(t)&=-\sin\left(t+\alpha\right).
  \end{aligned}
  \right.
\]
It follows that the switching functions are 
\begin{eqnarray*}
  \phiu(t) &=& 2\cos\left(\alpha\right) 
                -\cos\left(t+\alpha\right),\\
  \phiv(t) &=& -\cos\left(t+\alpha\right).
\end{eqnarray*}
The maximization condition \ref{PMP4} of PMP implies
\[
    \begin{aligned}
      \phiu(0)  & = \cos\left(\alpha\right)\le 0,\\
      \phiv(0)  & = -\cos\left(\alpha\right)\le 0,
    \end{aligned}
\]
i.e. 
$\alpha\in \{ \frac{\pi}{2},\frac{3\pi}{2}\}$
(in particular, the trajectory is abnormal). 
Note that $\phiu$ and $\phiv$ must be negative on a (small) open interval of the form $(0,\epsilon)$ since the starting control is the bang control $(-1,\vm)$.
Moreover,
since
$\phiu(0)=\phiv(0)=0$
and
\begin{eqnarray*}
  \dot{\phiu}(0)  &=&  \sin\alpha,\\ 
  \dot{\phiv}(0)  &=&  \sin\alpha, 
\end{eqnarray*}
$\phiu$ and $\phiv$ will be both negative on $(0,\epsilon)$ if and only if $\alpha=3\pi/2$.
Consequently,
there is only one extremal
(corresponding to $\alpha=3\pi/2$)
starting from $\Xzerot$ with control $(-1,\vm)$.
The two switching functions along this extremal
(\(
	\phiu(t)=\phiv(t) = -\sin t
\))
change their sign at $t=\pi$
and since
a $\uv$-switching cannot occur
(see Lemma \ref{lem:u-v-switchings_cannot_occur})
the trajectory \startmodif
$\traj{m}$ loses its optimality no later than time $\pi$.
\stopmodif

%
%
\subsection{Step 2: Analysis of the trajectories bifurcating from \texorpdfstring{$\traj{M}$}{}}
\label{S:Step2}

At this step we study separately all the bifurcating candidate extremals found at the previous step.

\subsubsection{Family 1: singular trajectory}

This family corresponds to a trajectory entering the turnpike at the time 
$\tsing=\pi-\asing$. 
From Lemma \ref{lem:Singular_control_and_switch},
the control $(0,\VM)$ remains unchanged and the trajectory exits the turnpike with the control 
$v = \VM$.
The trajectory is obtained from (\ref{eq:dubinstilde}) and lies on the turnpike, that is
\[
   \left\{
   \begin{aligned}
      \xt(t) &=-\VM t-\sqrt{\left(2\VM+\vm\right)}+\VM\tsing\\
      \yt(t) &=0
   \end{aligned}
   \right.,
   \quad\forall~t>\tsing.
\]
 It follows from (\ref{eq:adjoint_system}) that the covector is constant along a $u$-singular trajectory. Evaluating (\ref{eq:sol_adjoint}) at time $t=\tsing$ with $\alpha=\asing$ gives
\[
      \pp(t) = (\px(t),\py(t))=(-1,0), \quad \forall~t>\tsing.
\]
The $u$-singular trajectory is then extremal for any time $t>\tsing$.

\subsubsection{Family 2: first switching on $\boldsymbol{u}$}

This family corresponds to the set of trajectories switching from $\traj{M}$ with a $u$-switching. Recall that the trajectories of this family correspond to controls 
\(\left(1,\VM\right)\) and the $u$-switching time is
\(
\Tsw{MP}(\alpha) = -\alpha+\arccos\left(\cos\alpha({\VM+\vm})/{\VM}\right),
\)
with $\alpha\in\left(\frac{\pi}{2},\asing\right]$.
The trajectory, obtained from (\ref{eq:dubinstilde}), is
\[
   \left\{
   \begin{aligned}
     \xt(t) &=
     (\VM+\vm)\sin(t-2\Tsw{MP}(\alpha))-2\VM\sin(t-\Tsw{MP}(\alpha))
     \\   
     \yt(t) &=
     -\VM-(\VM+\vm)\cos(t-2\Tsw{MP}(\alpha))+2\VM\cos(t-\Tsw{MP}(\alpha)),
   \end{aligned}
   \right.
\]
for all $t>\Tsw{MP}(\alpha)$.
The covector satisfies (\ref{eq:adjoint_system}), which leads to the coordinates of the adjoint vector:
\[
   \left\{
   \begin{aligned}
    \px(t)  &=  \cos(t-\alpha-2\Tsw{MP}(\alpha))\\
    \py(t)  &=  \sin(t-\alpha-2\Tsw{MP}(\alpha)),
   \end{aligned}
   \right.
\]
Given the coordinates of the trajectory and the covector, the switching functions are:
\begin{eqnarray*}
    \phiu(t) & = & \py(t)\xt(t)-\px(t)\yt(t)  \\
             & = & \VM\cos(t-\alpha-2\Tsw{MP}(\alpha))+(\VM+\vm)\cos\alpha
                   -2\VM\cos(\alpha+\Tsw{MP}(\alpha)) \\
             & = & -(\VM+\vm)\cos\alpha+\VM\cos(t-\alpha-2\Tsw{MP}(\alpha))
    \\
    \phiv(t)  & = & -\px(t)         \\
              & = & -\cos(t-\alpha-2\Tsw{MP}(\alpha))
\end{eqnarray*}
As in the previous section, we study the sign of the switching functions.
Let $\zphiu{MP}(\alpha)$ and
$\zphiv{MP}(\alpha)$, both greater than $\Tsw{MP}(\alpha)$, be the first zeros of $\phiu$ and $\phiv$ respectively.
We thus have
\[
   \phiu(\zphiu{MP}(\alpha))
   =
   -(\VM+\vm)\cos\alpha+\VM\cos(\zphiu{MP}(\alpha)-\alpha-2\Tsw{MP}(\alpha))=0.
\]
Since we seek the first zero of $\phiu$, it follows that
\begin{align*}
  \zphiu{MP}(\alpha)
      &= \alpha+\arccos\left(\VMvmVMcos\right)
          +2\Tsw{MP}(\alpha),\\
      &= -\alpha +3\arccos\left(\VMvmVMcos\right).            
\end{align*}
In the same way, we determine the first zero of $\phiv$, we have
\[
   \phiv(\zphiv{MP}(\alpha))  =   
  -\cos(\zphiv{MP}(\alpha)-\alpha-2\Tsw{MP}(\alpha))=0.
\]
which gives the $v$-switching time
\begin{align*}
      \zphiv{MP}(\alpha)
      &=  -\frac{\pi}{2}+\alpha+2\Tsw{MP}(\alpha),     \\
      &=  -\frac{\pi}{2}-\alpha+2\arccos\left(\VMvmVMcos\right).
\end{align*}
Evaluating the difference between the two switching times, 
one can determine which switching function changes sign first.
Since $\asing<\pi$, the difference
\[
   \zphiu{MP}(\alpha)-\zphiv{MP}(\alpha)
   =\arccos\left(\VMvmVMcos\right)
    +\frac{\pi}{2}
\]
is positive for all 
$\alpha\in\left(\frac{\pi}{2},\asing\right]$.
The next switching is then a $v$-switching which occurs at a time 
$\Tsw{MPp}(\alpha)=\zphiv{MP}(\alpha)$
for \startmodif
this family of trajectories.
\stopmodif

\medskip
The curve made of $v$-switching points is called $v$-\emph{switching curve}. This curve, parametrized by $\alpha$, is given by
\[
   \left\{
   \begin{aligned}
      \xtswc{MP}(\alpha) & = \xt(\Tsw{MPp}(\alpha))
            = \left(\VM+\vm\right)\cos\alpha
      \\
      \ytswc{MP}(\alpha) & = \yt(\Tsw{MPp}(\alpha))
            = -\VM-(\VM+\vm)\sin\alpha
            +2\VM\sqrt{1-\left(\VMvmVMcos\right)^{2}},
   \end{aligned}
   \right.
\]
with $\alpha\in\left(\frac{\pi}{2},\asing\right]$.
On Fig. \ref{fig:1-switch}, this curve is shown in gray.

\subsubsection{Family 3: first switching on $\boldsymbol{v}$}

This family corresponds to the set of trajectories switching from $\traj{M}$ with a $v$-switching.
Recall that the trajectories of this family correspond to a control
$(-1,\vm)$ and a $v$-switching time $\Tsw{Mm}(\alpha)=-\alpha+3\pi/2$,
with $\alpha\in[\asing,\frac{3\pi}{2})$.
For all $t>\Tsw{Mm}(\alpha)$, the trajectory
is
\[
   \left\{
   \begin{aligned}
      \xt(t) &= (\VM-\vm)\cos(\alpha+t)-(\VM+\vm)\sin t
      \\
      \yt(t) &=  \vm-(\VM+\vm)\cos t-(\VM-\vm)\sin(\alpha+t)
   \end{aligned}
   \right.
\]
and the covector is given by
\[
   \left\{
   \begin{aligned}
      \px(t)
      &=\sin(t-\Tsw{Mm}(\alpha))=\cos(t+\alpha)\\
      \py(t)
      &=\cos(t-\Tsw{Mm}(\alpha))=-\sin(t+\alpha).
   \end{aligned}
   \right.
\]
Since the Family 3 of extremals results from a switching on $v$, according to Remark \ref{rem:2-v-switch-in-a-row}, the next switching is a $v$-switching. Moreover, since $\phi_{v}(t) = -\px(t)$ the next switching occurs at time 
$\Tsw{MmM}(\alpha) = \pi+\Tsw{Mm}(\alpha) = -\alpha+5\pi/2$.
As for the Family 2, we get the corresponding parametric equation of the switching curve:
\[
   \left\{
   \begin{aligned}
      \xtswc{Mm}(\alpha)
          &=   -(\VM+\vm)\cos\alpha\\
      \ytswc{Mm}(\alpha)
          &=  (2-\VM)-(\VM+\vm)\sin\alpha
   \end{aligned}
   \right.,\quad
   \alpha\in\left[\asing,\frac{3\pi}{2}\right).
\]
All results of Subsections \ref{S:Step1} and \ref{S:Step2} are shown in Fig. \ref{fig:1-switch}.
\begin{figure}[!ht]
\centering
  \def\svgwidth{0.6\columnwidth}
  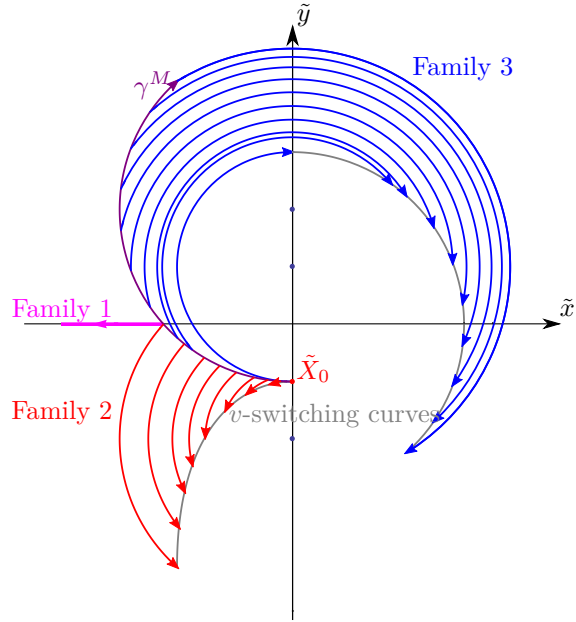
\protect\caption{Extremal trajectories starting from $\Xzerot$ and having one switching\label{fig:1-switch}}
\end{figure}

%
%
\section{Numerical simulations}\label{sec:applicationII}

%
%
\subsection{Synthesis in the \texorpdfstring{\UAV\!\!}{}-based coordinates}

Numerical simulations have been made 
with $\p=2$.
For the sake of clarity, the detailed construction of the synthesis is postponed to the Appendix.

\medskip
During the construction of the optimal synthesis some special curves appears, namely
\begin{itemize}
\item
\textit{switching curves},  i.e. curves made of switching points;

\item
\textit{cut loci}, i.e.
sets of points where the extremal curves of problem \textrm{\Ptwo} lose global optimality.
\end{itemize}
In practice, switching curves and cut loci can be very difficult to compute.
In the following, some of them were computed numerically.
Following the algorithm described in Section \ref{sec:algorithm}, the time-optimal synthesis corresponding to the problem \textrm{\Pone} has been solved.
The corresponding (discontinuous) state-feedback is given in Fig. \ref{fig:Final-synthesis} and Table \ref{tab:Colour-code} as a partition of the reduced state space.
\startmodif
\begin{remark}
Note that the first version of the synthesis (Fig. \ref{fig:Final-synthesis}) presented without detail in the conference paper \cite{LagacheSerresAndrieuCDC15} is mistaken.
Indeed, in \cite{LagacheSerresAndrieuCDC15}, the synthesis is incomplete since the last arc $\gamma^{MPpPMm}$ of Family 2 is missing.
\end{remark}
\stopmodif
\begin{remark}
Notice that the minimum time function (as a function of $\xt$ and $\yt$) is not continuous along the abnormal trajectory.
As a consequence,
the optimality of the synthesis cannot be confirmed a posteriori using the verification theorem
(\cite[Theorem 2.13]{PiccoliSussmann2000_SIAM})
based on the notion of regular synthesis
as it was done in
\cite{SoueresBalluchiBicchi2001}
for the case
of the Dubins' system for tracking a rectilinear route in minimum time.
Indeed,
it is easy to see that the minimum time function is not
weakly upper semicontinuous (w.u.s.c. for short)
and thus does not match the hypothesis of
\cite[Theorem 2.13]{PiccoliSussmann2000_SIAM}.
To see this let $\tilde X_n \to \Xzerot$ as $n\to\infty$ where
$\tilde X_n=(\xt_n,\yt_n)$,
with  $\xt_n>0$,
belongs to the $v$-switching curve
passing through $\Xzerot$.
Denote by $V(X)$ the minimum time to reach $\Xzerot$ from $X$.
We have
$V(\Xzerot)=0$. The $v$-switching occurring on the considered switching curve is a
switching from $\VM$ to $\vm$. 
Then, according to Remark \ref{rem:2-v-switch-in-a-row} and since 
$\phiv(t)=-\cos\left(\alpha+t\right)$,
\(
V(\tilde X_n)>\pi
\).
Consequently, $\lim_{n \to \infty}V(\tilde X_n)\geq\pi>0=V(\Xzerot)$ which shows that $V$ is not w.u.s.c. at $\Xzerot$.
\end{remark}
\begin{figure}
\centering
\def\svgwidth{0.8\columnwidth}
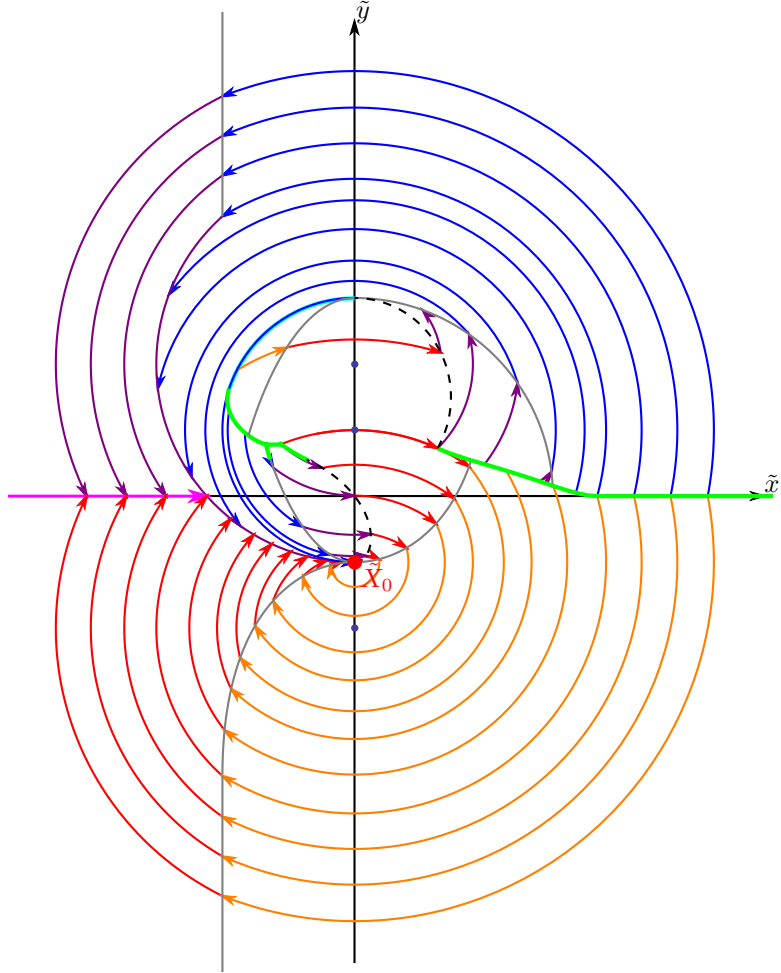
\caption{Time-optimal synthesis for the problem \textrm{\Pone}\label{fig:Final-synthesis}}
\end{figure}
\begin{table}
  \begin{centering}
    \begin{tabular}{|c|c|}
    \hline $(-1,\vm)$-bang arc & Blue\tabularnewline
    \hline $(1,\vm)$-bang arc & Orange\tabularnewline
    \hline $(-1,\VM)$-bang arc & Purple\tabularnewline
    \hline $(1,\VM)$-bang arc & Red\tabularnewline
    \hline {$u$-singular arc} & Magenta\tabularnewline
    \hline {$u$-switching curves} & Dashed black\tabularnewline
    \hline {$v$-switching curves} & Gray\tabularnewline
    \hline {Cut Locus} & Green\tabularnewline
    \hline {Abnormal Cut Locus} & Cyan\tabularnewline
    \hline 
    \end{tabular}
    \caption{Color convention of the optimal synthesis\label{tab:Colour-code}}
  \par\end{centering}
\end{table}
\subsection{Verification of the synthesis}
Using Lemma \ref{lem:one-switch-per-cadrant} for the problem {\textrm{\Pone}}, one can check the given synthesis. We have the functions:
\[
f(\qqt)=-\frac{\Delta_{Bu}(\qqt)}{\Delta_{A}(\qqt)}=-\frac{\yt}{\xt} \qquad \text{ and }\qquad
g(\qqt)=\frac{\Delta_{Bv}(\qqt)}{\Delta_{A}(\qqt)}=\frac{1}{\xt}.
\]
We define four domains according to the values of $f$ and $g$ (see Figure \ref{fig:Domains-verif}). We conclude that there is at most one switching on $u$ in each domains and at most one switching on $v$ in the union of domains 1 and 4 and in the union of 2 and 3. Moreover,
\begin{itemize}
\item in orthants 1 and 3: only a $u$-switching from 1 to -1 is allowed,

\item in orthants 2 and 4: only a $u$-switching from -1 to 1 is allowed,

\item in the union of orthants 1 and 4: only a $v$-switching from $\VM$ to $\vm$ is allowed,

\item in the union of orthants 2 and 3: only a $v$-switching from $\vm$ to $\VM$ is allowed.
\end{itemize}
One can easily check that the synthesis given in Figure \ref{fig:Final-synthesis} respects Lemma \ref{lem:one-switch-per-cadrant}.
\begin{figure}[!ht]
\centering
\def\svgwidth{0.5\columnwidth}
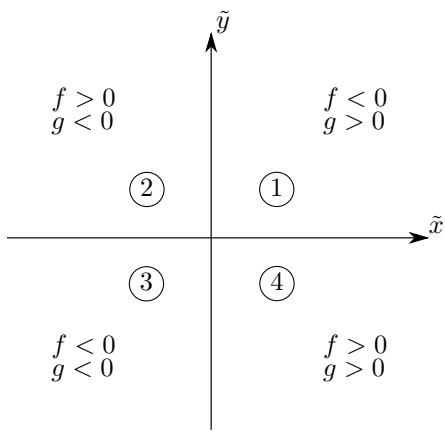
\caption{Domains on which at most one switching is possible\label{fig:Domains-verif}}
\end{figure}

%
%
\subsection{Correspondence with problem \texorpdfstring{\Pdub}{}}

The solutions of problem \textrm{\Pdub} can be deduced from the solutions of problem \textrm{\Pone}.
In this section,
we display pairs of figures
(Fig. \ref{fig:Optimal-traj-depart-cut} and \ref{fig:Optimal-traj-heart})
showing two solutions of problem \textrm{\Pone}
(that start from the same point in the cut locus)
and the corresponding lifted solutions of problem \textrm{\Pdub}.
Notice that the singular trajectory go straight to the center of the target.
\begin{figure}
  \centering
  \def\svgwidth{1\columnwidth}
  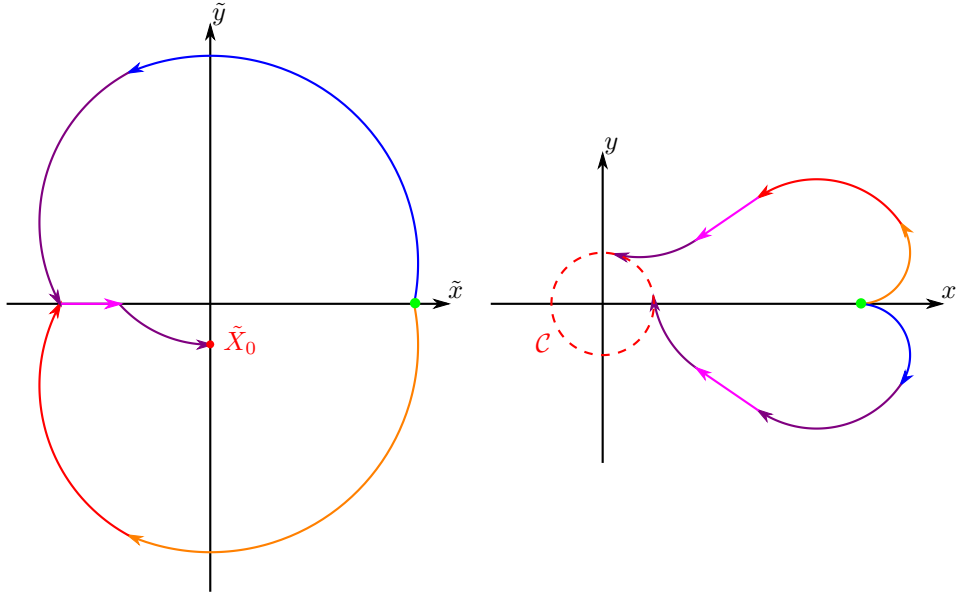
  \caption{Bang-bang-singular-bang optimal trajectories.
  Two optimal trajectories solutions of problem \textrm{\Pone} starting from the same point in the cut locus (left)
  and the corresponding optimal trajectories solutions of problem \textrm{\Pdub} (right)\label{fig:Optimal-traj-depart-cut}}
\end{figure}
\begin{figure}
  \centering
  \def\svgwidth{1\columnwidth}
  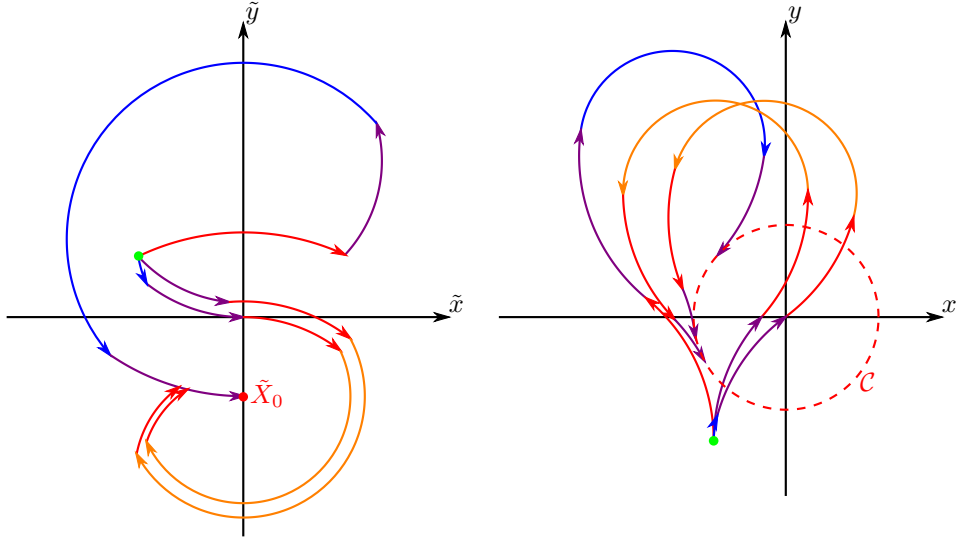
  \caption{Bang-bang-bang-bang-bang optimal trajectories and a bang-bang-bang-bang-bang-bang optimal trajectory.
  Three optimal trajectories solutions of problem \textrm{\Pone} starting from the same point in the cut locus (left)
  and the corresponding optimal trajectories solutions of problem \textrm{\Pdub} (right)\label{fig:Optimal-traj-heart}}
\end{figure}
\begin{figure}
  \centering
  \def\svgwidth{1\columnwidth}
  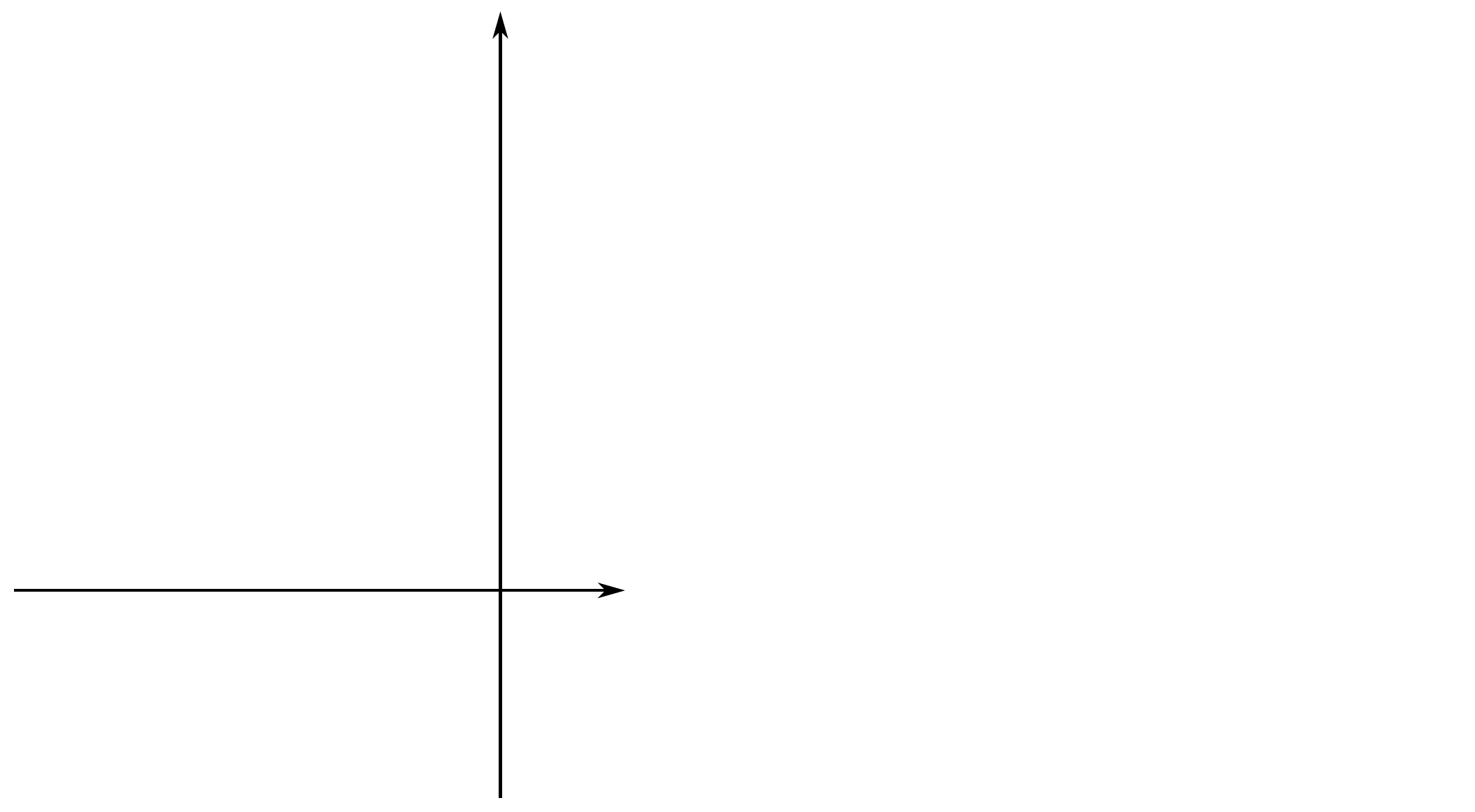
  \caption{The abnormal trajectory of problem \textrm{\Pone} and the corresponding trajectory for problem \textrm{\Pdub} (right)\label{fig:Optimal-traj-abnormal}}
\end{figure}

%
%
\subsection{Stability of the optimal feedback control}

We have shown that the discontinuous feedback control law is a time-optimal feedback for system (\ref{eq:dubinstilde}).
However, the optimality of the feedback does not imply Lyapunov stability
(see e.g. the beautiful and simple example in \cite{Boscain2014}.
In our case we have the following proposition.

\begin{proposition}
The time-optimal feedback law renders system (\ref{eq:dubinstilde}) globally asymptotically stable.
\end{proposition}
\begin{proof}
The convergence to the equilibrium $\Xzerot$ follows from the optimality of the synthesis
(actually, the equilibrium is even reached in finite time).
We prove the stability by a direct application of Lyapunov's definition (in the reduced space).
A time-optimal trajectory starting in the ball
\(
\{Z \mid \|Z-\Xzerot\| \le \delta \le \vm\}
\)
may escape this ball if and only if it contains a bang arc corresponding to the control
\(
(u,v)=(1,\VM).
\)
It is an easy matter to see that for such a trajectory
the norm
\(
\|X(t;Z)-\Xzerot\|
\)
increases only along its $(1,\VM)$-bang arc.
Moreover,
this bang arc stops at the switching curve defined by \eqref{SC:MPp}, which is above the horizontal line
\(
\{(\xt,\yt)\in\R \mid \yt=-\vm\}.
\)
Hence,
one easily check that for every $\delta \le \vm$
\[
  \sup_{\|Z-\Xzerot\|\le\delta}
  \sup_{t\in \R_+}
  \left\{
  \|X(t;Z)-\Xzerot\|\right\}
  \le \sqrt{\delta(2(\VM-\vm)+\delta)}.
\]
Consequently,
for every $\epsilon>0$,
choosing either
\(
\delta=
\vm
\)
if $\epsilon \ge (2(\VM-\vm)+\vm)$
or
\(
\delta=
\sqrt{((\VM-\vm)^2+\epsilon^2)} - (\VM-\vm)
\)
otherwise
we obtain
\[
\forall~\epsilon>0 \quad
\exists~\delta>0 \mid
\|X(t;Z)-\Xzerot\|\le\epsilon, \quad
\forall~t\ge0, \quad
\forall~Z\in B(\Xzerot,\delta).
\]
\end{proof}

%
%
\section{Conclusion}

\startmodif
In this paper we have solved a time-minimal control problem for a kinematic model describing a \UAV
flying at constant altitude with controls on the steering angle and on the linear velocity.
\stopmodif
Thanks to a change of coordinates applied to the three-dimensional Dubins' system,
we could simplify the problem and use (and extend) the existing theory of time-optimal syntheses
for two-dimensional single input affine control systems
to two-dimensional affine control systems with two inputs.
We gave the time-optimal synthesis as a state-feedback law such that the target is reached optimally
in finite time.
\startmodif
Note however that our kinematic model does not actually correspond to a realistic dynamic, as the velocities may not be continuous. A more challenging problem would be to consider a model with controls on accelerations. 
\stopmodif

%
%
\section*{Appendix: construction details of the synthesis}\label{S:Appendix}
\renewcommand{\theequation}{A.\arabic{equation}}
\renewcommand{\thesection}{A}

%
%
\subsection{3rd arc}

%
%
\subsubsection{Family 1}

In this subsection, we study the trajectories bifurcating from the singular trajectory 
$\traj{Ms}$. A trajectory exiting the turnpike has either the control $(-1,\VM)$ or the control $(1,\VM)$. 

From each point 
\[
   \left\{
   \begin{aligned}
      \xt(\tau) &=-\VM\tau-\sqrt{\left(2\VM+\vm\right)}+\VM\tsing\\
      \yt(\tau) &=0
   \end{aligned}
   \right.,
   \quad\forall~\tau>\tsing,
\]
two trajectories bifurcate. This family of trajectories is then parametrized by $\tau$ instead of $\alpha$. For $u\in\{-1,1\}$ and for all $t>\tau$ the corresponding trajectory is
\[
   \left\{
   \begin{aligned}
      \xt(t) &=
      \left(-\sqrt{(2\VM+\vm)}-\VM \tau + \VM \tsing\right) \cos(t-\tau)
      -\VM \sin(t-\tau)
      \\
      \yt(t) &=
      u\left(-\VM +
            \left(-\sqrt{(2\VM+\vm)}-\VM \tau + \VM \tsing\right) \sin(t-\tau)
            +\VM\cos(t-\tau)\right),
   \end{aligned}
   \right.
\]
the covector is
\[
   \left\{
     \begin{aligned}
       \px(t) &=  -u\cos(t-\tau)  \\
       \py(t) &=  -\sin(t-\tau),
     \end{aligned}
   \right.
\]
the switching functions, defined by \eqref{eq:switching_function_u} and \eqref{eq:switching_function_v}, are
\begin{eqnarray*}
    \phi_{u}(t) 
     & = & 
     \VM - \VM \cos(t-\tau),   \\
     \phi_{v}(t)
     & = & 
     u\cos(t-\tau).
\end{eqnarray*}
Since in each case $\phiu > 0$, the next switching occurs on the control $v$ at time $\Tsw{MsPp}(\tau)=\Tsw{MsMm}(\tau)=\tau+\pi/2$.

The parametric equations of the corresponding switching curve is
\[
   \left\{
   \begin{aligned}
      \xtswc{Msi}(\tau)
      &=
      -\VM                           \\
      \ytswc{Msi}(\tau)
      &=
      u(-\VM -\sqrt{(2\VM+\vm)}-\VM\tau + \VM \tsing),
   \end{aligned}
   \right.
\]
with $i=\cP$ (resp. $\cM$) for $u=1$ (resp. $-1$).

%
%
\subsubsection{Family 2}

In this subsection, we study the trajectories $\traj{MPp}$ switching from $\traj{MP}$ 
at a time $\Tsw{MPp}(\alpha)$ with a control $(1,\vm)$ and $\alpha\in\left(\frac{\pi}{2},\asing\right]$. These trajectories start from the parametrized curve
\begin{equation}\label{SC:MP}
   \left\{
   \begin{aligned}
   \xtswc{MP}(\alpha) &=  (\VM+\vm)\cos\alpha
   \\
   \ytswc{MP}(\alpha) &=  -\VM-(\VM+\vm)\sin\alpha
                   +2\VM\sqrt{1-\left(\VMvmVMcos\right)^{2}},
   \end{aligned}
   \right.
\end{equation}
and, for all $t \ge \Tsw{MPp}(\alpha)$, have coordinates
\begin{equation}\label{Traj:MPp}
   \left\{
   \begin{aligned}
      \xt(t) &=
          \xtswc{MP}(\alpha) \cos(t-\Tsw{MPp}(\alpha))-
          (\vm+\ytswc{MP}(\alpha))\sin(t-\Tsw{MPp}(\alpha))
      \\
      \yt(t) &=
       (\vm+\ytswc{MP}(\alpha))\cos(t-\Tsw{MPp}(\alpha))+
      \xtswc{MP}(\alpha)\sin(t-\Tsw{MPp}(\alpha))-\vm 
   \end{aligned}
   \right.
\end{equation}
with $\Tsw{MPp}(\alpha)=-\pi/2-\alpha+2\arccos\left(\VMvmVMcos\right)=-\frac{\pi}{2}+\alpha+2\Tsw{MP}(\alpha)$
, i.e.
\[
   \left\{
   \begin{aligned}
      \xt(t) &=
        \!\begin{multlined}[t] (\VM-\vm)\cos\left(t+\alpha-2\acos\right) \\[.1em]
                - 2\VM\sin\left(t+\alpha-\acos\right) \\[.1em]
                +(\VM+\vm) \sin\left(t+2\alpha-2\acos\right)
        \end{multlined}\\[.2em]
      \yt(t) &= 
       \!\begin{multlined}[t] -\vm -(\VM-\vm)\sin\left(t+\alpha-2\acos\right) \\[.1em]
                + 2\VM\cos\left(t+\alpha-\acos\right) \\[.1em]
                -(\VM+\vm)\cos\left(t+2\alpha-2\acos\right).
    \end{multlined}
   \end{aligned}
   \right.
\]
By continuity of the covector, we have $\pp(\Tsw{MPp}(\alpha)) = (0,1)$,
the covector is then given by
\[
   \left\{
   \begin{aligned}
    \px(t)  &=  -\sin(t-\Tsw{MPp}(\alpha))\\
    \py(t)  &=  \cos(t-\Tsw{MPp}(\alpha))
   \end{aligned},
   \qquad\forall~t \geq \Tsw{MPp}(\alpha)
   \right..
\]

Since the previous switching was a $v$-switching from $\VM$ to $\vm$, the next switching will be a $v$-switching (see Remark \ref{rem:2-v-switch-in-a-row}) and will occur after a duration $\pi$ since ($\phiu(t)=\px(t)$ is a sinusoidal function). We then have 
\(
  \Tsw{MPpP}(\alpha)= \Tsw{MPp}(\alpha)+\pi
        =\pi/2-\alpha+2\arccos\left(\VMvmVMcos\right)
\).\\
The switching curve is then given by
\[
   \left\{
   \begin{aligned}
      \xtswc{MPp}(\alpha)
          &=   \xt(\Tsw{MPpP}(\alpha))=\xt(\Tsw{MPp}(\alpha)+\pi)\\
      \ytswc{MPp}(\alpha)
          &=  \yt(\Tsw{MPpP}(\alpha))=\yt(\Tsw{MPp}(\alpha)+\pi).
   \end{aligned}
   \right.
\]
which, using \eqref{Traj:MPp}, becomes
\begin{equation}\label{SC:MPp}
   \left\{
   \begin{aligned}
      \xtswc{MPp}(\alpha) &=
          -\xtswc{MP}(\alpha)
      \\
      \ytswc{MPp}(\alpha) &=
      -2-\ytswc{MP}(\alpha),
   \end{aligned}
   \right.
\end{equation}
where $\xtswc{MP}$ and $\ytswc{MP}$ are given by \eqref{SC:MP}.

%
%
\subsubsection{Family 3}

In this subsection, we study the trajectories $\traj{MmM}$ switching from $\traj{Mm}$ 
at a time $\Tsw{MmM}(\alpha) = \pi+\Tsw{Mm}(\alpha) = -\alpha+5\pi/2$ with a control 
$(-1,\VM)$ and $\alpha\in\left[\asing,\frac{3\pi}{2}\right)$. These trajectories starts from the parametrized curve
\[
   \left\{
   \begin{aligned}
      \xtswc{Mm}(\alpha)
          &=   -(\VM+\vm)\cos\alpha\\
      \ytswc{Mm}(\alpha)
          &=  (2-\VM)-(\VM+\vm)\sin\alpha.
   \end{aligned}
   \right.
\]
For all $t \geq \Tsw{MmM}(\alpha)$, the trajectories are
\[
   \left\{
   \begin{aligned}
      \xt(t) &= \xtswc{Mm}(\alpha)\cos(-t+\Tsw{MmM}(\alpha))
      +(\VM - \ytswc{Mm}(\alpha))\sin(-t+\Tsw{MmM}(\alpha))
      \\
      \yt(t) &=  \VM - (\VM - \ytswc{Mm}(\alpha))\cos(-t+\Tsw{MmM}(\alpha))
          +\xtswc{Mm}(\alpha)\sin(-t+\Tsw{MmM}(\alpha)),
   \end{aligned}
   \right.
\]
which, after simplifications, yields
\[
   \left\{
   \begin{aligned}
      \xt(t) &= -(\VM+\vm)\sin t+
      2(\VM-\vm)\cos(t+\alpha)
      \\
      \yt(t) &=  \VM - 
          2(\VM-\vm)
          \sin(t+\alpha)-
          (\VM+\vm)\cos t.
   \end{aligned}
   \right.
\]
By continuity, $\pp(\Tsw{MmM}(\alpha))  =  (0,1)$
and for all $t \geq \Tsw{MmM}(\alpha)$ the covector is
\[
   \left\{
   \begin{aligned}
      \px(t)
      &=\sin(t-\Tsw{MmM}(\alpha))=\sin(t+\alpha-5\pi/2)=\cos(t+\alpha)\\
      \py(t)
      &=\cos(t-\Tsw{MmM}(\alpha))=\cos(t+\alpha-5\pi/2)=-\sin(t+\alpha).
   \end{aligned}
   \right.
\]
The switching functions, defined by \eqref{eq:switching_function_u} and \eqref{eq:switching_function_v}, are then
\begin{eqnarray*}
    \phiu(t) 
    &=& (\VM+\vm)\cos\alpha - \VM\cos(t+\alpha),
    \\
    \phiv(t) 
     & = & - \cos(t+\alpha).
\end{eqnarray*}
The analysis of these switching functions shows that 
\begin{itemize}
  \item for $\alpha \in [\asing,2\pi-\arccos(-\VM/(\VM+\vm))]$ $\phiu$ is always negative. $\phiv$ starts with positive values and changes sign at time 
  $\Tsw{MmMm}(\alpha)=\Tsw{MmM}(\alpha)+\pi$, at which a $v$-switching occurs.
  \item for $\alpha \in [2\pi-\arccos(-\VM/(\VM+\vm)),3\pi/2]$ $\phiu$ starts with negative values and $\phiv$ with positive values. The function that has a zero first is
  $\phiu$ and there is a $u$-switching at time 
  $\Tsw{MmMP}(\alpha)=\Tsw{MmM}(\alpha) - \arcsin((\VM+\vm)/\VM\cos\alpha)$
\end{itemize}

For $\alpha \in [2\pi-\arccos(-\VM/(\VM+\vm)),3\pi/2]$, the switching curve is given by
\[
   \left\{
   \begin{aligned}
      \xtswc{MmM}(\alpha) &=
          \!\begin{multlined}[t]
          -(\VM+\vm)\cos\left(-\alpha-\asin\right) 
                    +2\VMCmvmCVMcos
          \end{multlined}
          \\[.2em]
      \ytswc{MmM}(\alpha) &= 
          \!\begin{multlined}[t]
              \VM - 
              2(\VM-\vm)\cos\left(\asin\right) \\[.1em]
              +(\VM+\vm)\sin\left(-\alpha-\asin\right).
          \end{multlined}
   \end{aligned}
   \right.
\]

We will see later that the switching curve for the other interval of $\alpha$ is not needed (see Remark \ref{rem:no_MmMm}, page \pageref{rem:no_MmMm}).

%
%
\subsubsection{Reduction of $\traj{MPp}$, $\traj{Mm}$ and $\traj{MmM}$}

The plot of the trajectories computed previously (see left part of Fig. \ref{fig:Cut_right_before_after}) shows that $\traj{MPp}$ intersects with $\traj{Mm}$ and $\traj{MmM}$. We then have to compute when the optimality of each trajectory is lost. This computation has been made numerically 
for $\p = 2$ and indicates that
\begin{itemize}
  \item for $\alpha \in [\atraj{MPp},\asing]$, with $\atraj{MPp} \simeq 2.19947$, the trajectories $\traj{MPp}$ loose their optimality. $\traj{MPpP}$ is then defined for $\alpha \in (\pi/2,\atraj{MPp})$.
  \item for $\alpha \in [\asing,\atraj{MmM}]$, with $\atraj{MmM} \simeq 3.84506$, the trajectories $\traj{Mm}$ or $\traj{MmM}$ loose their optimality. $\traj{MmMP}$ is then defined for $\alpha \in (\atraj{MmM},3\pi/2)$.
\end{itemize}
\begin{figure}[!ht]
  \centering
  \def\svgwidth{1\columnwidth}
  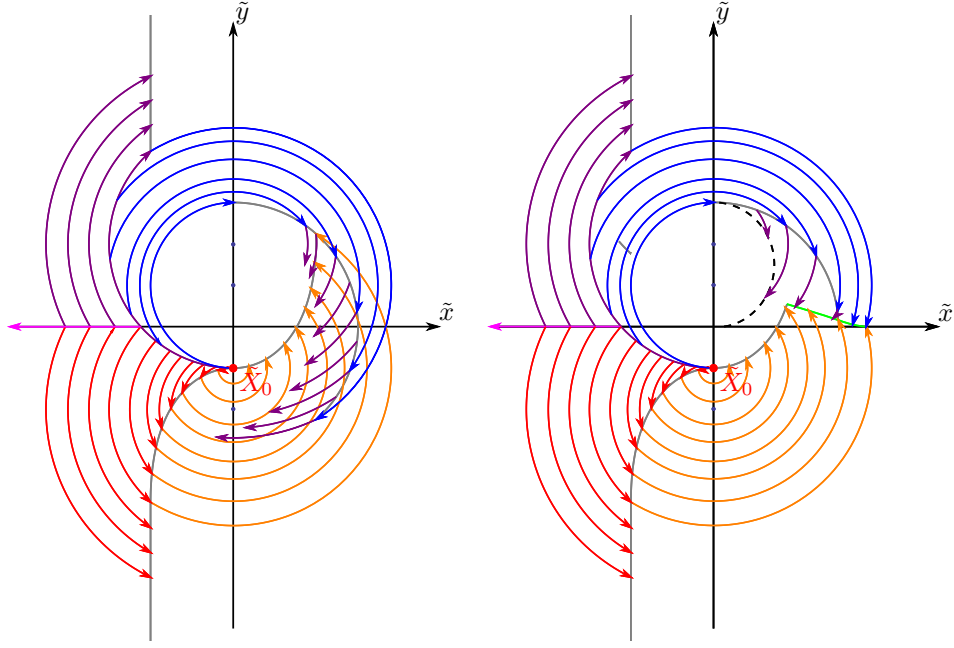
  \caption{Intersection of $\traj{MPp}$ with $\traj{Mm}$ and $\traj{MmM}$ (left) and the corresponding cut locus (right)}\label{fig:Cut_right_before_after}
\end{figure}
The times at which trajectories loose optimality, although not detailed here, allow to plot the corresponding cut locus (in green on the right part of Fig. \ref{fig:Cut_right_before_after}).
The rest of the synthesis will be restricted to the above-mentioned intervals.

%
%
\subsection{4th arc}

%
%
\subsubsection{Family 1}

In the following, the computations are made for $\traj{MsPp}$. Computations for $\traj{MsMm}$ are  obtained similarly since these two trajectories are symmetric with respect to the $\xt$ axis.
The trajectories $\traj{MsPp}$ bifurcate at the time $\Tsw{MsPp}(\tau)=\tau+\pi/2$ from 
\[
   \left\{
   \begin{aligned}
      \xtswc{MsP}(\tau) &= 
           -\VM                                                               \\
      \ytswc{MsP}(\tau) &= 
           -\VM -\sqrt{(2\VM+\vm)}-\VM\tau + \VM \tsing.
   \end{aligned}
   \right.
\]
Then for all $t\geq\Tsw{MsPp}(\tau)$ 
\[
   \left\{
   \begin{aligned}
      \xt(t) &=
            \xtswc{MsP}(\tau) \cos(t-\Tsw{MsPp}(\tau))
            -(\vm + \ytswc{MsP}(\tau))\sin(t-\Tsw{MsPp}(\tau))
      \\
      \yt(t) &= -\vm + (\vm + \ytswc{MsP}(\tau)) \cos(t-\Tsw{MsPp}(\tau))
            +\xtswc{MsP}(\tau) \sin(t-\Tsw{MsPp}(\tau)),
   \end{aligned}
   \right.
\]
which leads to
\[
   \left\{
   \begin{aligned}
      \xt(t) &=
            -\VM \sin(t-\tau) + \\ & 
            (\vm-\VM -\sqrt{(2\VM+\vm)}-\VM \tau + \VM \tsing)
            \cos(t-\tau)
      \\
      \yt(t) &= -\vm
            +\VM \cos(t-\tau) + \\ & 
            (\vm-\VM -\sqrt{(2\VM+\vm)}-\VM\tau + \VM \tsing)
            \sin(t-\tau).
   \end{aligned}
   \right.
\]
By continuity of the covector $\pp(\Tsw{MsPp}(\tau)) =  (0,-1)$
and the covector is
\[
   \left\{
     \begin{aligned}
       \px(t) &= \sin(t-\Tsw{MsPp}(\tau))=-\cos(t-\tau)   \\
       \py(t) &= -\cos(t-\Tsw{MsPp}(\tau))= -\sin(t-\tau) 
     \end{aligned}
   \right.,
    \quad\forall~t\geq\Tsw{MsPp}(\tau).
\]
Since the last switching was a $v$-switching from $\VM$ to $\vm$, the next switching will be a $v$-switching. Taking into account the form of $\phiv$, the switching will occur after a duration $\pi$. We conclude that $\Tsw{MsPpP}(\tau)=\Tsw{MsPp}(\tau)+\pi=\tau+3\pi/2$.\\
The plot of the trajectories shows that trajectories $\traj{MsMm}$ and $\traj{MsPp}$ will necessarily intersect before reaching their next switching curve, which, therefore does not need to be computed.\\
Moreover, since for a same $\tau$, $\traj{MsMm}(t)$ and $\traj{MsPp}(t)$ are symmetric with respect to the $\xt$ axis, they will reach the $\xt$ axis at the same time. We obtain that the set
\[
  \left\{(\xt,\yt) \mid \xt > \VM^2-\vm+\left(\vm-\VM-\sqrt{(2\VM+\vm)}\right)^2 \text{ and } \yt = 0 \right\}
\]
forms a cut locus.

%
%
\subsubsection{Family 2}

Trajectories of Family 2 bifurcate from $\traj{MPp}$ at time 
$\Tsw{MPpP}(\alpha)=\pi/2-\alpha+2\acos$, with control $(1,\VM)$. Notice that after reduction of this family, we have $\alpha \in (\pi/2,\atraj{MPp})$.
$\traj{MPp}$ starts from the switching curve
\[
   \left\{
   \begin{aligned}
    \xtswc{MPp}(\alpha) &=  -\left(\VM+\vm\right)\cos\alpha
   \\
   \ytswc{MPp}(\alpha) &=  (\VM-2) + (\VM+\vm)\sin\alpha
                   -2\VM\sqrt{1-\left(\VMvmVMcos\right)^{2}},
   \end{aligned}
   \right.
\]
and has, for all $t\geq\Tsw{MPpP}(\alpha)$, coordinates
\[
   \left\{
   \begin{aligned}
      \xt(t) &=
          \xtswc{MPp}(\alpha) \cos(t-\Tsw{MPpP}(\alpha))-
          (\VM+\ytswc{MPp}(\alpha))\sin(t-\Tsw{MPpP}(\alpha))
      \\
      \yt(t) &=
      -\VM + (\VM+\ytswc{MPp}(\alpha))\cos(t-\Tsw{MPpP}(\alpha))+
      \xtswc{MPp}(\alpha)\sin(t-\Tsw{MPpP}(\alpha)).
   \end{aligned}
   \right.
\]
which is equivalent to
\[
   \left\{
   \begin{aligned}
      \xt(t) &= 
      \!\begin{multlined}[t]
        2(\VM-\vm)\cos\left(t+\alpha-2\acos\right) \\[.1em]
                - 2\VM\sin\left(t+\alpha-\acos\right) \\[.1em]
                +(\VM+\vm)\sin\left(t+2\alpha-2\acos\right)
      \end{multlined}
      \\[.2em]
      \yt(t) &= 
      \!\begin{multlined}[t]
         -\VM + 2(\VM-\vm)\sin\left(t+\alpha-2\acos\right) \\[.1em]
                + 2\VM\cos\left(t+\alpha-\acos\right) \\[.1em]
                -(\VM+\vm)\cos\left(t+2\alpha-2\acos\right).
      \end{multlined}          
   \end{aligned}
   \right.
\]
The covector is then
\[
   \left\{
     \begin{aligned}
       \px(t) &= -\sin(t-\Tsw{MPpP}(\alpha))   \\
       \py(t) &= \cos(t-\Tsw{MPpP}(\alpha)),
     \end{aligned}
   \right.
\]
which leads to
\[
   \left\{
     \begin{aligned}
       \px(t) &= \cos\left(t+\alpha-2\acos\right)   \\
       \py(t) &= \sin\left(t+\alpha-2\acos\right),
     \end{aligned}
   \right.
\]
The switching functions, defined by \eqref{eq:switching_function_u} and \eqref{eq:switching_function_v}, are then
\begin{eqnarray*}
     \phi_{u}(t) 
     & = & 
     -(\VM + \vm)\cos\alpha - \VM \cos\left(t+\alpha-2\acos\right),
     \\
     \phi_{v}(t) 
     & = & 
     \cos\left(t+\alpha-2\acos\right).
\end{eqnarray*}
The analysis of these switching functions for $\alpha \in (\pi/2,\atraj{MPp})$ shows that the next switching will be a $u$-switching that will occur at time
\[
\Tsw{MPpPM}(\alpha) = \Tsw{MPpP}(\alpha) - \asin  
                        = -\alpha+3\acos.
\]
The switching curve is then
\[
   \left\{
   \begin{aligned}
    \xtswc{MPpP}(\alpha) &= 
            \VMvmVMcos \\ &
            \left(
            -3\VM\cos\left(\asin\right) + 2(\VM-\vm)\sin\alpha+(\VM+\vm)\sin\alpha
            \right)
   \\
   \ytswc{MPpP}(\alpha) &= 
      \!\begin{multlined}[t]    
            - 3\VM 
            + (\VM+\vm)\cos\left(\asin\right)\sin\alpha \\[.1em]
            + 3\VM\left(\VMvmVMcos\right)^2
            + 2(\VM-\vm)\cos\left(\asin\right).
      \end{multlined}
   \end{aligned}
   \right.
\]

%
%
\subsubsection{Reduction of $\traj{MPpP}$ and $\traj{MmM}$}

The simulation shows that $\traj{MPpP}$ and $\traj{MmM}$ intersect themselves. The corresponding cut locus is computed numerically and it yields:

\begin{itemize}
  \item for $\alpha \in [\atraj{MPpP},\asing]$, with $\atraj{MPpP} \simeq 2.18628$, the trajectories $\traj{MPpP}$ loose their optimality. $\traj{MPpPM}$ is then defined for $\alpha \in (\pi/2,\atraj{MPpP})$.
  \item for $\alpha \in [\asing,\atraj{MmM'}]$, with $\atraj{MmM'} \simeq 4.09691$, the trajectories $\traj{MmM}$ loose their optimality. $\traj{MmMP}$ is then defined for $\alpha \in (\atraj{MmM'},3\pi/2)$.
\end{itemize}

The remaining part of the synthesis will be restricted to the above-mentioned intervals.

\begin{remark}\label{rem:no_MmMm}
  Recall that the trajectories $\traj{MmMm}$ are defined for 
  $\alpha \in [\asing,2\pi-\arccos(-\VM/(\VM+\vm))]$.
  Since $\VM = 2$,
  $\atraj{MmM'} > 2\pi-\arccos(-\VM/(\VM+\vm))$, these arcs do not appear because $\traj{MmM}$ looses its optimality before switching to $\traj{MmMm}$.
\end{remark}

%
%
\subsubsection{Family 3}

In this subsection, we study the trajectories $\traj{MmMP}$ switching from 
$\traj{MmM}$ at a time 
$\Tsw{MmMP}(\alpha) = -\alpha+5\pi/2-\arcsin((\VM+\vm)/\VM\cos\alpha)$ with a control $(1,\VM)$ and $\alpha \in (\atraj{MmM'},3\pi/2)$. The starting points are on the parametrized curve

\[
   \left\{
   \begin{aligned}
      \xtswc{MmM}(\alpha) &=   
      \!\begin{multlined}[t]
         -(\VM+\vm)\cos\left(-\alpha-\asin\right) 
          +2\VMCmvmCVMcos
      \end{multlined}
          \\[.2em]
      \ytswc{MmM}(\alpha) &=  
      \!\begin{multlined}[t]
          \VM - 
          2(\VM-\vm)\cos\left(\asin\right) \\[.1em]
          +(\VM+\vm)\sin\left(-\alpha-\asin\right).
      \end{multlined}
   \end{aligned}
   \right.
\]

For all $t\geq\Tsw{MmMP}(\alpha)$, the trajectory
is
\[
   \left\{
   \begin{aligned}
      \xt(t) &=
          \xtswc{MmM}(\alpha) \cos(t-\Tsw{MmMP}(\alpha))-
          (\VM+\ytswc{MmM}(\alpha))\sin(t-\Tsw{MmMP}(\alpha))
      \\[.2em]
      \yt(t) &=
      \!\begin{multlined}[t]
        -\VM + (\VM+\ytswc{MmM}(\alpha))\cos(t-\Tsw{MmMP}(\alpha)) \\[0.1em]
        +\xtswc{MmM}(\alpha)\sin(t-\Tsw{MmMP}(\alpha)).
        \end{multlined}
   \end{aligned}
   \right.
\]
By continuity of the covector we have
\[
   \left\{
   \begin{aligned}
    \px(\Tsw{MmMP}(\alpha))  &=  -\left(\VMvmVMcos\right)\\
    \py(\Tsw{MmMP}(\alpha))  &=  \cos\left(\asin\right).
   \end{aligned}
   \right.
\]
The covector is then given by
\[
   \left\{
   \begin{aligned}
      \px(t)  &= 
             \cos\left(t+\alpha-2\acos\right)
      \\
      \py(t)  &=
              \sin\left(t+\alpha-2\acos\right)
   \end{aligned}.
   \right.
\]

The switching functions, defined by \eqref{eq:switching_function_u} and \eqref{eq:switching_function_v}, are then
\[
\begin{aligned}
    \phiu(t)  
             &= 
        \!\begin{multlined}[t]
             \cos\left(\asin\right)\xtswc{MmM}(\alpha) 
             + \left(\VMvmVMcos\right)
                (\VM+\ytswc{MmM}(\alpha)) \\[.1em]
             - \VM \cos\left(t+\alpha-2\acos\right),
        \end{multlined}            
    \\
    \phiv(t) 
     & = -\cos\left(t+\alpha-2\acos\right).
\end{aligned}
\]

The analysis of these switching functions shows that the next switching will be a 
$v$-switching that will occur at time 
\[
  \begin{aligned}
  \Tsw{MmMPp}(\alpha)   &= -\alpha +5\pi/2-2\asin \\
                        &= \Tsw{MmMP}(\alpha) - \asin.
  \end{aligned}
\]

This switching time leads to the switching curve:
\[
   \left\{
   \begin{aligned}
      \xtswc{MmMP}(\alpha) &=
          \xtswc{MmM}(\alpha) \cos\left(\asin\right)\\ &
          -(\VM+\ytswc{MmM}(\alpha))\left(\VMvmVMcos\right)
      \\
      \ytswc{MmMP}(\alpha) &=
        -\VM + (\VM+\ytswc{MmM}(\alpha))\cos\left(\asin\right) \\&
        +\xtswc{MmM}(\alpha)\left(\VMvmVMcos\right).
   \end{aligned}
   \right.
\]

%
%
\subsubsection{Family 2}

In this subsection we study the trajectories $\traj{MPpPM}$ switching from 
$\traj{MPpP}$ at time $\Tsw{MPpPM} = -\alpha +3 \acos$ with control 
$(-1,\VM)$ and $\alpha \in (\pi/2,\atraj{MPpP})$. We recall
\[
   \left\{
   \begin{aligned}
    \xtswc{MPpP}(\alpha) &= 
            \VMvmVMcos \\ &
            \left(
            -3\VM\cos\left(\asin\right) + 2(\VM-\vm)\sin\alpha+(\VM+\vm)\sin\alpha
            \right)
   \\
   \ytswc{MPpP}(\alpha) &= 
      \!\begin{multlined}[t]    
            - 3\VM 
            + (\VM+\vm)\cos\left(\asin\right)\sin\alpha \\[.1em]
            + 3\VM\left(\VMvmVMcos\right)^2
            + 2(\VM-\vm)\cos\left(\asin\right).
      \end{multlined}
   \end{aligned}
   \right.
\]
The trajectories are then
\[
   \left\{
   \begin{aligned}
      \xt(t) &=
          \xtswc{MPpP}(\alpha) \cos(t-\Tsw{MPpPM}(\alpha))
          -(\VM-\ytswc{MPpP}(\alpha))\sin(t-\Tsw{MPpPM}(\alpha))
      \\[.2em]
      \yt(t) &=
      \!\begin{multlined}[t]
      \VM - (\VM-\ytswc{MPpP}(\alpha))\cos(t-\Tsw{MPpPM}(\alpha)) \\[.1em]
      -\xtswc{MPpP}(\alpha)\sin(t-\Tsw{MPpPM}(\alpha)).
      \end{multlined}
   \end{aligned}
   \right.
\]
The covector satisfies
\[
   \left\{
     \begin{aligned}
       \px(\Tsw{MPpPM}(\alpha)) &= \left(\VMvmVMcos\right)   \\
       \py(\Tsw{MPpPM}(\alpha)) &= \sin\left(\acos\right),
     \end{aligned}
   \right.
\]
which leads to

\[
   \left\{
   \begin{aligned}
      \px(t)  &= 
             \cos\left(t+\alpha-4\acos\right) \\
      \\
      \py(t)  &= 
             -\sin\left(t+\alpha-4\acos\right).
   \end{aligned}
   \right.
\]

We have then the switching functions, defined by \eqref{eq:switching_function_u} and \eqref{eq:switching_function_v},
\[
   \begin{aligned}
     \phi_{u}(t)  
      &=
      \!\begin{multlined}[t]
      \xtswc{MPpP}(\alpha)\sqrt{1-\left(\frac{\VM + \vm}{\VM}\cos\alpha\right)^2} + 
      (\VM-\ytswc{MPpP}(\alpha))\frac{\VM + \vm}{\VM}\cos\alpha
      \\
      - \VM \cos\left(t+\alpha-4\acos\right),
      \end{multlined}
     \\
     \phi_{v}(t) 
     & =
     -\cos\left(t+\alpha-4\acos\right).
   \end{aligned}
\]
The analysis of these switching functions shows that the next switching will be a 
$v$-switching and will occur at time 
\[
  \begin{aligned}
    \Tsw{MPpPMm}(\alpha)  &= \Tsw{MPpPM}(\alpha)-\asin \\
                          &=  -\alpha -\frac{\pi}{2} +4 \acos.
  \end{aligned}
\]
The switching curve is then
\[
   \left\{
   \begin{aligned}
    \xtswc{MPpPM}(\alpha) &=  
            \xtswc{MPpP}(\alpha) \sqrtacos
          +(\VM-\ytswc{MPpP}(\alpha))\VMvmVMcos
   \\[.2em]
   \ytswc{MPpPM}(\alpha) &= 
   \!\begin{multlined}[t]
         \VM - (\VM-\ytswc{MPpP}(\alpha))\sqrtacos \\[.1em]
        +\xtswc{MPpP}(\alpha)\VMvmVMcos.
   \end{multlined}
   \end{aligned}
   \right.
\]

%
%
\subsubsection{Reduction of $\traj{MPpP}$ and $\traj{MPpPM}$}

The simulation shows that $\traj{MPpP}$ and $\traj{MPpPM}$ intersect themselves. Notice that this is a particular case because these trajectories belong to the same family (see e.g. Fig. \ref{fig:Cut_left_before}). The corresponding cut locus is computed numerically and it yields that for 
$\alpha \in [\atraj{MPpPM},\atraj{MPpP})$, with $\atraj{MPpPM} = 2.13033$, 
the trajectories $\traj{MPpP}$ loose their optimality. 
$\traj{MPpPMm}$ is then defined for $\alpha \in (\pi/2,\atraj{MPpPM})$.\\

The rest of the synthesis will be restricted to the above-mentioned intervals.

%
%
\subsection{5th arc:}

%
%
\subsubsection{Family 3}

In this subsection we study the trajectories $\traj{MmMPp}$ switching from 
$\traj{MmMP}$ at time $\Tsw{MmMPp} =-\alpha +5\pi/2-2\asin$ with control 
$(1,\vm)$ and $\alpha \in (\atraj{MmM'},3\pi/2)$. We recall that these trajectories start from

\[
   \left\{
   \begin{aligned}
      \xtswc{MmMP}(\alpha) &=
      \!\begin{multlined}[t]
          \xtswc{MmM}(\alpha) \cos\left(\asin\right) \\[.1em]
          -(\VM+\ytswc{MmM}(\alpha))\left(\frac{\VM + \vm}{\VM}\cos\alpha\right)
      \end{multlined}
      \\[.2em]
      \ytswc{MmMP}(\alpha) &=
      \!\begin{multlined}[t]
        -\VM + (\VM+\ytswc{MmM}(\alpha))\cos\left(\asin\right) \\[.1em]
        +\xtswc{MmM}(\alpha)\left(\frac{\VM + \vm}{\VM}\cos\alpha\right),
      \end{multlined}
   \end{aligned}
   \right.
\]

For all $t \geq \Tsw{MmMPp}(\alpha)$, the trajectories are
\[
   \left\{
   \begin{aligned}
      \xt(t) &=
      \!\begin{multlined}[t]
          \xtswc{MmMP}(\alpha) \cos(t-\Tsw{MmMPp}(\alpha)) \\[.1em]
          -(\vm+\ytswc{MmMP}(\alpha))\sin(t-\Tsw{MmMPp}(\alpha))
      \end{multlined}
      \\[.2em]
      \yt(t) &=
      \!\begin{multlined}[t]
        -\vm + (\VM+\ytswc{MmMP}(\alpha))\cos(t-\Tsw{MmMPp}(\alpha)) \\[.1em]
        +\xtswc{MmMP}(\alpha)\sin(t-\Tsw{MmMPp}(\alpha)),
      \end{multlined}
   \end{aligned}
   \right.
\]



Since the previous switching occurred on $v$, the next switching will also be on $v$ and will occur after a time $\pi$ (following the same reasoning as previously).\\
At this step of the construction of the synthesis, we immediately see that the next switching will not occur.  Indeed, these trajectories will necessarily intersect with trajectories found earlier in the construction of the synthesis before reaching their switching curve.

%
%
\subsubsection{Family 2}

In this subsection we study the trajectories $\traj{MPpPMm}$ switching from 
$\traj{MPpPM}$ at time $\Tsw{MPpPMm} = -\alpha -\frac{\pi}{2} +4 \acos$ with control 
$(-1,\vm)$ and $\alpha \in (\pi/2,\atraj{MPpPM})$.

The trajectories are
\[
   \left\{
   \begin{aligned}
      \xt(t) &=
      \!\begin{multlined}[t]
          \xtswc{MPpPM}(\alpha) \cos(t-\Tsw{MPpPMm}(\alpha))\\[.1em]
          -(\vm-\ytswc{MPpPM}(\alpha))\sin(t-\Tsw{MPpPMm}(\alpha))
          \end{multlined}
      \\[.2em]
      \yt(t) &=
      \!\begin{multlined}[t]
      \vm - (\vm-\ytswc{MPpPM}(\alpha))\cos(t-\Tsw{MPpPMm}(\alpha)) \\[.1em]
      -\xtswc{MPpPM}(\alpha)\sin(t-\Tsw{MPpPMm}(\alpha)).
      \end{multlined}
   \end{aligned}
   \right.
\]

Since the previous switching occurred on $v$, the next switching will also be on $v$ and will occur after a time $\pi$ (following the same reasoning as previously).\\
At this step of the construction of the synthesis, we immediately see that the next switching will not occur.  Indeed, these trajectories will necessarily intersect with trajectories found earlier in the construction of the synthesis before reaching their switching curve.

%
%
\subsubsection{Reduction of $\traj{MmMP}$, $\traj{MmMPp}$, $\traj{MPpPM}$ and $\traj{MPpPMm}$}

The simulation shows that $\traj{MPpPMm}$ intersects with $\traj{MmMP}$, $\traj{MmMPp}$ and $\traj{MPpPM}$ (see Fig. \ref{fig:Cut_left_before}). The corresponding cut locus is computed numerically.

It also appears that $\traj{MmMPp}$ intersects with the abnormal trajectory $\traj{m}$. \startmodif
Clearly, $\traj{MmMPp}$ looses
\stopmodif its optimality at this intersection. A part of the abnormal trajectory is then a cut locus (in cyan on Fig. \ref{fig:Final-synthesis}).

\begin{figure}[!ht]
\centering
\def\svgwidth{0.5\columnwidth}
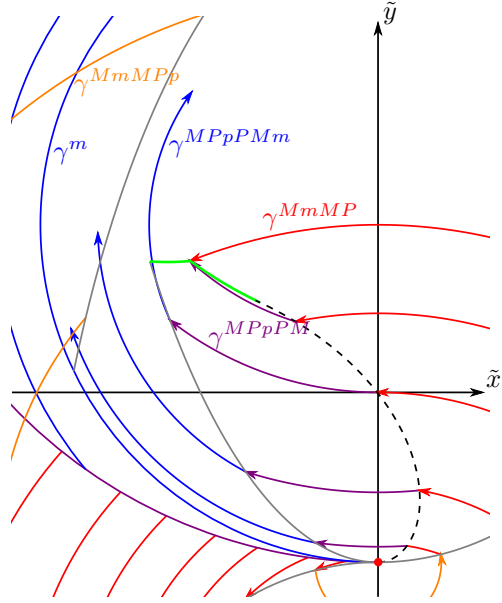
\caption{Intersection of $\traj{MPpPMm}$ with $\traj{MmMP}$, $\traj{MmMPp}$ and $\traj{MPpPM}$ \label{fig:Cut_left_before}}
\end{figure}

Since no new trajectories can be generated, the construction of the synthesis stops.


%
%
\bibliographystyle{AIMS}
\bibliography{./biblio} 

\medskip
Received xxxx 20xx; revised xxxx 20xx.
\medskip

\end{document}